\documentclass[12pt]{amsart}
\pdfoutput=1 

\usepackage{ifpdf}

\usepackage{amsmath}
\usepackage{amsthm}
\usepackage{amsfonts}
\usepackage{amssymb}

\ifpdf
    \usepackage[colorlinks,citecolor=blue,linkcolor=blue]{hyperref}
\fi

\usepackage{amscd}	
\usepackage{cancel}	
\usepackage{microtype} 

\usepackage{graphicx}
\usepackage{color}

\ifdefined\draftversion
    \providecommand{\showkeys}{}
    \usepackage[shortalphabetic,initials,backrefs]{amsrefs}
\else
    \usepackage[shortalphabetic,initials]{amsrefs}
\fi

\ifdefined\showkeys
	\usepackage[color]{showkeys} 
\fi

\makeatletter
\def\blfootnote{\xdef\@thefnmark{}\@footnotetext}
\makeatother

\IfFileExists{.git/git.tex}{
\input{.git/git.tex}

}
{

}

\ifpdf
  \newcommand{\fig}[2]{
    \IfFileExists{#1.pdf_tex}{
      \def\svgwidth{#2}\input{#1.pdf_tex}
    }{
      \frame{Missing figure ``#1.pdf\_tex''}
      \message{LaTeX Warning: Missing figure ``#1.pdf\_tex'' on input line \the\inputlineno}
    }
  }
\else
    \IfFileExists{#1.eps_tex}{
      \def\svgwidth{#2}\input{#1.eps_tex}
    }{
      \frame{Missing figure ``#1.eps\_tex''}
      \message{LaTeX Warning: Missing figure ``#1.eps\_tex'' on input line \the\inputlineno}
    }
  }
\fi

\newcommand{\dimension}{n}

\newcommand{\dist}{\operatorname{dist}}

\newcommand{\interior}{\operatorname{int}}
\newcommand{\trace}{\operatorname{tr}}

\newcommand{\esssup}{\operatorname*{ess\,sup}}
\newcommand{\essinf}{\operatorname*{ess\,inf}}

\newcommand{\argmax}{\operatorname*{arg\,max}}

\newcommand{\divo}{\operatorname{div}}

\newcommand{\dx}{\;dx}

\newcommand{\abs}[1]{\left|#1\right|}
\newcommand{\pth}[1]{\left(#1\right)}
\newcommand{\bra}[1]{\left[#1\right]}
\newcommand{\set}[1]{{\left\{#1\right\}}}

\newcommand{\at}[2]{{{\left.{#1}\right|}_{#2}}}

\newcommand{\norm}[1]{\left\|#1\right\|}

\newcommand{\cl}[1]{\overline{#1}}	

\newcommand{\e}{\ensuremath{\varepsilon}}

\newcommand{\R}{\ensuremath{\mathbb{R}}}

\newcommand{\Rd}{\ensuremath{{\mathbb{R}^{\dimension}}}}
\newcommand{\Rn}{\Rd}
\newcommand{\Z}{\ensuremath{\mathbb{Z}}}
\newcommand{\N}{\ensuremath{\mathbb{N}}}

\newcommand{\T}{\ensuremath{\mathbb{T}}}
\newcommand{\Tn}{{\T^\dimension}}

\newcommand{\td}[2]{\frac{d {#1}}{d {#2}}}

\definecolor{grey}{rgb}{0.6,0.6,0.6}

\numberwithin{equation}{section}

\newtheorem{theorem}{Theorem}[section]
\newtheorem{lemma}[theorem]{Lemma}
\newtheorem{proposition}[theorem]{Proposition}
\newtheorem{corollary}[theorem]{Corollary}

\newtheorem{definition}[theorem]{Definition}

\theoremstyle{definition}

\newtheoremstyle{remarkstyle}
  {3pt}
  {3pt}
  {\small}
  {}
  {\bfseries}
  {.}
  { }
  {}

\theoremstyle{remarkstyle}
\newtheorem{remark}[theorem]{Remark}
\newtheorem{example}[theorem]{Example}
\usepackage[margin=1in]{geometry}
\usepackage{bbm}
\newcommand{\CH}{{CH}}

\newcommand{\domain}{\mathcal D}
\newcommand{\Wulff}{\mathcal W}
\newcommand{\TT}{\mathcal{T}}
\newcommand{\SW}{W^{\rm sl}}
\newcommand{\SE}{E^{\rm sl}}
\newcommand{\sign}{\operatorname{sign}}
\newcommand{\aff}{\operatorname{aff}}
\newcommand{\facet}{A}

\title[Approximation by regular facets]{Approximation of general facets by regular facets with respect to anisotropic total variation
energies and its application to the crystalline mean curvature flow}

\author[Y. Giga]{Yoshikazu Giga}
\author[N. Po\v{z}\'{a}r]{Norbert Po\v{z}\'{a}r}

\date{\today}

\subjclass[2010]{}

\begin{document}

\begin{abstract}
  We show that every bounded subset of an Euclidean space can be approximated by a set that admits
  a certain vector field, the so-called Cahn-Hoffman vector field, that is subordinate to a given
  anisotropic metric and has a square-integrable divergence.  More generally, we introduce a
  concept of facets as a kind of directed sets, and show that they can be approximated in a similar
  manner.

  We use this approximation to construct test functions necessary to prove the comparison principle
  for viscosity solutions
  of the level set formulation of the crystalline mean curvature flow that were recently
  introduced by the authors.  As a consequence, we obtain the well-posedness of the
  viscosity solutions in an arbitrary dimension, which extends the validity of the result in the
  previous paper.
\end{abstract}

\maketitle

\section{Introduction}

In this note we consider the question of approximating general compact subsets of an
Euclidean space by sets whose boundary regularity is similar to the regularity of
the Wulff shape of a given anisotropic norm. This regularity is related to the existence of the
so-called Cahn-Hoffman vector field with an $L^2$ divergence on the set.

To give a more specific example of what we have in mind, let $n \geq 1$ be a dimension, $\facet \subset \Rn$ be a compact set and
$\rho > 0$. Then there exists a compact set $G \supset \facet$ with a smooth boundary such that
$d_H(\partial G, \partial \facet) < \rho$, and a Lipschitz
continuous vector field $z: G \to \Rn$ such that $\abs{z(x)} \leq 1$, $z \cdot \nu = 1$ on $\partial
G$. Here $\nu$ is the unit outer normal vector on $\partial G$ and $d_H$ is the Hausdorff distance
with respect to the Euclidean norm. For this result see
\cite{GGP14JMPA}. Such a vector field is an example of a Cahn-Hoffman vector field for the
isotropic metric given by the Euclidean norm $\abs\cdot$. The unit ball $\set{z: |z| \leq 1}$ is
its Wulff shape. Such vector fields are important for
evaluating the anisotropic curvature of a surface or the subdifferential of a total variation
energy. Unfortunately,
no Cahn-Hoffman vector field might exist for a given set, in which case the anisotropic curvature cannot be
evaluated. However, the above result allows us to approximate arbitrary sets by sets for which the
anisotropic curvature is well-defined.

The goal of the present note is to extend this result to general anisotropic norm-like functions
(non necessarily symmetric) $W: \Rn \to [0, \infty)$. We assume that $W$ is (i) convex, (ii)
positively one-homogeneous, $W(tp) = t W(p)$ for all $p \in \Rn$, $t \geq 0$, and (iii) positive
definite, $W(p) = 0$ if and only if $p = 0$. We will call such a function $W$ an \emph{anisotropy}.
Given $W$, its polar $W^\circ$ is defined as
\begin{align}
  \label{polar}
  W^\circ(x) := \sup \set{x \cdot p: p \in \Rn,\ W(p) \leq 1}.
\end{align}
$W^\circ$ is also an anisotropy, see \cite{Rockafellar}.
The \emph{Wulff shape $\Wulff$} of $W$ is given as the set
\begin{align*}
  \Wulff := \set{x \in \Rn: W^\circ(x) \leq 1}.
\end{align*}
By the positive definiteness of $W^\circ$, $\Wulff$ is a compact set containing $0$ in its
interior.
We claim that any compact set $\facet$ can be approximated by a larger set $G$ with the properties as
above but with respect to the anisotropy $W$. However, we need to relax the
notion of boundary values of $z$. We will show that for every $\rho > 0$ there exists a compact set
$G \supset \facet$, $d_H(\partial G, \partial \facet) < \rho$, an open set $U \supset G$, a Lipschitz continuous function $\psi$ on $U$, $\psi
\geq 0$, $\psi(x) = 0$ if and only if $x \in G$, and a vector field $z \in L^\infty(U, \Rn)$ with
$\divo z \in L^2(U)$ such that $z \in \Wulff$, $z \cdot \nabla \psi = W(\nabla \psi)$ a.e. on
$U$.
Note that the last condition on the vector field $z$ is equivalent to requiring $z \in \partial
W(\nabla \psi)$ a.e. on $U$, where $\partial W$ is the subdifferential of $W$ with respect to the
Euclidean inner product $x \cdot y$, given as
\begin{align}
  \label{subdifferential-W}
  \partial W(p) := \set{\xi \in \Rn: W(p + h) - W(p) \geq \xi \cdot h, \ h \in \Rn} = \set{\xi \in
  \Rn: \xi \in \Wulff,\ \xi \cdot p = W(p)}.
\end{align}

As we will see later, the approximation discussed so far is related to the surfaces of convex or concave
solid bodies. To handle general solid bodies,
we want to control the direction in which the vector field $z$ flows
through the boundary of the set $\facet$. Therefore we introduce a concept of \emph{directionality} of a set. In
particular, we specify which components of the complement are sources and which are sinks of the
vector field. We will call such directed sets \emph{facets}.

\begin{definition}
  \label{de:facet}
Let $n \geq 0$ be the dimension. A compact set $\facet \subset \Rn$ together with a direction $\chi \in
C(\facet^c, \set{-1, 1})$ is called an \emph{$n$-dimensional facet}, and we write it as $(\facet, \chi)$.
We will set $\chi(x) = 0$ for $x \in \facet$.
\end{definition}

The direction $\chi$ introduces an order on the family of all facets: $(\facet, \chi) \preceq (G, \tilde
\chi)$ if $\chi \leq \tilde\chi$. If $\chi, \tilde \chi \geq 0$, we have $(\facet, \chi) \preceq (G,
\tilde \chi)$ if and only if $\facet \subset G$. We now introduce ``regular'' facets that admit a certain
vector field subordinate to the anisotropy $W$.

\begin{definition}
  \label{de:cahn-hoffman-facet}
We say that $(\facet, \chi)$ is a
\emph{$W^\circ$-$(L^2)$ Cahn-Hoffman facet} if there exists an open set $U
\subset \Rn$, a Lipschitz function $\psi$ on $U$ and a vector field $z \in L^\infty(U, \Rn)$, $\divo
z \in L^2(U)$, such that $\facet
\subset U$, $\sign \psi = \chi$ on $U$, $z(x) \in \partial W(\nabla \psi(x))$ for a.e. $x \in U$.
Here we define $\sign s$ to be $-1, 0, 1$ if $s < 0$, $s = 0$ or $s > 0$, respectively.
\end{definition}

If $\facet$ has a Lipschitz boundary, we would like to say that $(\facet, \chi)$ is a $W^\circ$-$(L^2)$ Cahn-Hoffman facet if there exists a vector field $z \in
L^\infty(\facet,\Rn)$, $\divo z \in L^2(\facet)$ such that $z \in \Wulff$ on $\facet$ and $z(x) \in \partial
W(\chi(x+\nu 0+) \nu(x))$ on $\partial \facet$, where $\chi(x + \nu0+)$ is the limit of $\chi$ at $x$
from the direction outside of $\facet$ and $\nu(x)$ is the unit outer normal to $\facet$ at $x$.
However, we currently do not know if this is equivalent to
Definition~\ref{de:cahn-hoffman-facet}.

Such vector fields $z$ are usually referred to as Cahn-Hoffman vector fields. We are mainly interested in
the divergence of these vector fields, in particular in the one that is minimal in the $L^2$-norm
on $\facet$. We call such an $L^2$-function
the \emph{$W^\circ$-$(L^2)$ minimal divergence of the facet $(\facet, \chi)$},
Definition~\ref{de:minimal-divergence-facet}.
We will see later, Proposition~\ref{pr:Lambda-comparison}, that this minimal divergence is unique
and depends only on $\facet$ and $\chi$.

The main claim of this paper is that an arbitrary $n$-dimensional facet $(\facet, \chi)$ can be
approximated by a
Cahn-Hoffman facet $(G, \tilde \chi)$ consistently with the direction, that is, $(\facet, \chi) \preceq (G,
\tilde \chi)$.

\begin{theorem}
  \label{th:main}
  Let $(\facet, \chi)$ be an $n$-dimensional facet and $W$ an anisotropy.
  Given $\rho > 0$ there
  exists a $W^\circ$-$(L^2)$ Cahn-Hoffman facet $(G, \tilde \chi)$ such that $\chi(x) \leq \tilde \chi(x) \leq
  \sup_{\abs{x - y} \leq \rho}\chi(y)$ for $x \in \Rn$.
\end{theorem}

The main application of the approximation result Theorem~\ref{th:main} is the recently
developed notion of viscosity solutions of the level set formulation of a crystalline curvature
flow by the authors in \cite{GP16}.
In particular, the viscosity solutions of the initial value problem
\begin{align}
  \label{level-set}
  \left\{
    \begin{aligned}
      u_t + F(\nabla u, \divo (\nabla W)(\nabla u)) &= 0, && \text{in $\Rn \times (0,
      \infty)$},\\
      \at{u}{t = 0} &= u_0, && \text{in $\Rn$},
    \end{aligned}
  \right.
\end{align}
were introduced. Here $F: \Rn \times \R \to \R$ is assumed to be a continuous function that is
nonincreasing in the second variable, and the initial data $u_0$ is assumed to be continuous and
constant outside of a compact set. Finally, $W$ is a \emph{crystalline} anisotropy, that is, a
piece-wise linear anisotropy.

This problem can be derived as the level set formulation of the motion of a set
$\set{E_t}_{t \geq 0}$ by the
\emph{crystalline} mean curvature flow
\begin{align}
  \label{curvature-flow}
V = f(\nu, \kappa_W),
\end{align}
where $V$ is the normal velocity of the
surface $\partial E_t$, $\nu$ is its outer unit normal vector and $\kappa_W = -\divo_{\partial E_t} (\nabla W)(\nu)$ is the
anisotropic (crystalline) mean curvature of $\partial E_t$. The function $f: \mathcal S^{n-1}
\times \R \to \R$ is assumed to be continuous and nondecreasing in the second variable.
The idea of the level set method is to introduce an auxiliary function $u: \Rn \times [0, \infty) \to
\R$ whose every sub-level set $t \mapsto \set{x: u(x, t) < c}$, $c \in \R$, evolves under the
crystalline mean
curvature flow. This function then satisfies the initial value problem \eqref{level-set} with an
appropriate $u_0$, typically so that $E_0 = \set{x: u_0(x) < 0}$.
Since $V = -\frac{u_t}{|\nabla u|}$, $\nu = \frac{\nabla u}{|\nabla u|}$ and $\divo_{\partial
E_t}(\nabla W)(\nu) = \divo (\nabla W)(\nabla u)$, see \cite{G06}, we deduce that $F(p,
\xi) = |p| f(\frac{p}{|p|}, -\xi)$.
In \cite{GP16}, the viscosity solutions of \eqref{level-set} were shown to satisfy the comparison
principle in three dimensions, and the well-posedness of this problem was established.

Theorem~\ref{th:main} allows us to extend the result of \cite{GP16} to an arbitrary dimension. For
a detailed discussion of the notion of viscosity solutions for \eqref{level-set} see
Section~\ref{sec:visc-sol}. Heuristically, the key idea is to define the viscosity solution using
the comparison principle with test functions that have flat parts that are $W^\circ$-$(L^2)$
Cahn-Hoffman facets, and then interpret the operator $\divo(\nabla W)(\nabla u)$ as the
$W^\circ$-$(L^2)$ minimal divergence of the facet. This is consistent with the theory of monotone
operators when \eqref{level-set} is in a divergence form. The approximation result in
Theorem~\ref{th:main} implies that the family of these test functions is sufficiently large.  We
have the following results.

\begin{theorem}[Comparison principle]
  \label{th:comparison-principle-intro}
  Let $W : \Rn \to \R$ be a crystalline anisotropy, and $F \in
  C(\Rn \times \R)$ be nonincreasing in the second variable, $F(0, 0) = 0$. Suppose that $u$ is an
  upper semicontinuous function and $v$ is a lower semicontinuous function, $u$ and $v$ are
  a viscosity
  subsolution and a viscosity supersolution, respectively, of \eqref{level-set} in $\Rn \times (0, T)$ for some $T >
  0$, and that there exist a compact set $K \subset \Rn$ and constants $c_u, c_v \in \R$ such that $u
  \equiv c_u \leq c_v \equiv v$ on $(\Rn \setminus K) \times [0,T]$. Then $u(\cdot, 0) \leq v(\cdot, 0)$ on $\Rn$
  implies $u \leq v$ on $\Rn \times (0, T)$.
\end{theorem}

\begin{theorem}[Stability]
  \label{th:stability-intro}
  Let $W : \Rn \to \R$ be a crystalline anisotropy, and $F \in
  C(\Rn \times \R)$ be nonincreasing in the second variable.
  Let $u_m$, $m \in \N$, be a sequence of viscosity solutions of \eqref{level-set} with $W =
  W_m$ and initial data $u_{0,m}$. Suppose that the sequence $u_{0,m}$ is locally uniformly bounded.
  If either
  \begin{enumerate}
    \item $W_m \in C^2(\Rn)$, $a_m^{-1} \leq \nabla^2 W_m \leq a_m$ for some $a_m > 0$, and $W_m
      \searrow W$, or
    \item $W_m$ are smooth anisotropies, that is, $W_m$ is an anisotropy, $W_m \in C^2(\Rn \setminus \set0)$ and $\set{p:
      W_m(p) \leq 1}$ is strictly convex, such that $W_m \to W$ locally uniformly,
  \end{enumerate}
  then
  \begin{align*}
    \bar u(x,t) = \limsup_{(y, s, m) \to (x, t, \infty)} u_m(y,s) \quad \text{and} \quad
    \underline u(x,t) = \liminf_{(y, s, m) \to (x, t, \infty)} u_m(y,s)
  \end{align*}
  are respectively a viscosity subsolution and
   a viscosity supersolution of
   \begin{align*}
     u_t + F(\nabla u, \divo (\nabla W)(\nabla u)) = 0 \quad \text{in }\Rn \times (0, \infty).
   \end{align*}
\end{theorem}

\begin{theorem}[Well-posedness]
  \label{th:well-posedness-intro}
  Let $W : \Rn \to \R$ be a crystalline anisotropy, and $F \in
  C(\Rn \times \R)$ be nonincreasing in the second variable, $F(0, 0) = 0$.
  Suppose that $u_0 \in C(\Rn)$ and
  that there exist a compact set $K \subset \Rn$ and a constant $c \in \R$ such that $u_0 = c$ on
  $\Rn \setminus K$.
  Then there exists a unique viscosity solution of \eqref{level-set} such that $u(x, t) = c$  on
  $\Rn \setminus K_t$ for some compact set $K_t$, $t \geq 0$.
\end{theorem}

Note that $F(p, \xi) = |p| f(\frac{p}{|p|}, - \xi)$ coming from the level set formulation of the crystalline mean curvature flow
\eqref{curvature-flow} satisfies automatically $F(0,0) = 0$. In fact, $F(0, \xi) = 0$ for all $\xi
\in \R$. Therefore the above theorem implies the global unique existence (up to fattening of the
set $\set{x: u(x,t) = 0}$) of the crystalline mean curvature flow
for any initial bounded set $E_0$.

\begin{theorem}
  \label{th:global-ex-flow}
  Let $W : \Rn \to \R$ be a crystalline anisotropy, and $f \in
  C(\mathcal S^{n-1} \times \R)$ be nondecreasing in the second variable.
  For every bounded open set $E_0 \subset \Rn$ there exists a unique evolution $\set{E_t}_{t\geq
  0}$ of the crystalline mean curvature flow $V = f(\nu, \kappa_W)$.
\end{theorem}

\subsection{Relationship to the anisotropic total variation energy}
In this section we relate the above approximation result to the problem of evaluating the
subdifferential of a total variation energy functional.
Let $W$ be an anisotropy and let $\Omega$ be either a smooth domain in $\Rn$ or the torus $\Tn:= \Rn \setminus \Z^n$.
Consider the anisotropic total variation energy functional $E: L^2(\Omega) \to \R$
defined as
\begin{align}
  \label{TV-energy}
  E(\psi) :=
    \begin{cases}
    \int_{\Omega} W(D \psi) \dx, & \psi \in BV(\Omega) \cap L^2(\Omega),\\
    +\infty, & \text{otherwise}.
    \end{cases}
\end{align}
Here $D \psi$ is the distributional
gradient of a function of bounded variation, whose space is denoted as $BV(\Omega)$.
In general, $D\psi$ is a vector-valued Radon measure. Therefore $E$ is understood as the relaxation
(closure or lower semicontinuous envelope) of the energy \eqref{TV-energy} defined for
$W^{1,1}(\Omega)$ functions.
It is known that this relaxation is equivalent to defining $W(D\psi)$ as the
measure
\begin{align*}
  W(\nabla \psi) L^n + W\pth{\frac{D^s \psi}{\abs{D^s \psi}}} \abs{D^s \psi},
\end{align*}
where we decompose $D\psi$ into to the absolutely continuous part $\nabla \psi L^n$ and
the singular part $D^s \psi$, with respect to the $n$-dimensional Lebesgue measure $L^n$.
$\frac{D^s \psi}{\abs{D^s \psi}}$ denotes the Radon-Nikod\'ym derivative.

The subdifferential $\partial E: L^2(\Omega) \to 2^{L^2(\Omega)}$ is defined as the set-valued mapping
\begin{align}
  \label{E-subdiff}
  \partial E(\psi) := \set{v \in L^2(\Omega): E(\psi + h) - E(\psi) \geq (v, h),\ h \in L^2(\Omega)},
\end{align}
where $(v, h) := \int_{\Omega} v h \dx$ is the $L^2$-inner product.
$\partial E(\psi)$ is a closed convex, possibly empty set. The domain of the subdifferential is
defined as $\domain(\partial E) := \set{\psi \in L^2(\Omega): \partial E(\psi) \neq \emptyset}$.

The characterization of $\partial E$ is well-known \cite{Moll,ACM}. Following \cite{Anzellotti},
we introduce the set of bounded vector fields with $L^2$ divergence
\begin{align*}
  X_2(\Omega) := \set{z \in L^\infty(\Omega; \Rn): \divo z \in L^2(\Omega)}.
\end{align*}
For $\psi \in Lip(\Omega)$, the set $\partial E(\psi)$ can be characterized as
\begin{align*}
  \partial E(\psi) = \set{\divo z: z \in X_2(\Omega),\ z \in \partial W(\nabla \psi)\
  \text{a.e.},\ [z\cdot \nu] = 0 \text{ on }\partial \Omega},
\end{align*}
where $\partial W$ is the subdifferential of $W$ with respect to the inner product on $\Rn$,
introduced in \eqref{subdifferential-W}. The symbol $[z \cdot \nu]$ denotes the boundary trace of
a vector field in $X_2(\Omega)$, see \cite{Anzellotti}. The
vector fields $z$ are usually called Cahn-Hoffman vector fields, and we shall denote their set as
\begin{align}
  \label{ch}
  \CH(\psi) := \set{z \in X_2(\Omega):\ z \in \partial W(\nabla \psi) \text{ a.e.}}, \qquad \psi
  \in Lip(\Omega).
\end{align}
Note that $\partial E(\psi)$ might be empty even for smooth $\psi$. For example, consider the function
$\psi(x) = \sin(2\pi x_1)$ on $\Tn$ with anisotropy $W(p) = \norm{p}_2 := \pth{\sum_{i=1}^n
\abs{p_i}^2}^{1/2}$, or $W(p) = \norm{p}_1 := \sum_{i=1}^n \abs{p_i}$. In this
case,
every possible candidate vector field $z$ with $z(x) \in \partial W(\nabla \psi)$ a.e. will have a
jump discontinuity in $z_1$ across the sets $\set{x_1 =
\frac 14}$, where $\psi$ attains its maximum, and $\set{x_1 = \frac 34}$, where it attains its
minimum.  This is a serious difficulty for using the operator
$\partial E$ in the theory of viscosity solutions,
which usually rely on evaluating the differential operator on a class of sufficiently smooth test
functions.

Theorem~\ref{th:main}
allows us to approximate any function by a function with nonempty $\partial E$ arbitrarily close in
the Hausdorff distance of
the respective positive and negative sets. Stated in terms of the functions $\psi$, we get the
following theorem.

\begin{theorem}
  \label{th:approximate_pair}
  Let $W$ be an anisotropy and $\Omega$ be either $\Tn$ or a bounded domain with Lipschitz boundary in $\Rn$.
  If $\psi$ is a Lipschitz function on $\Omega$ with a compact zero set, for every $\rho > 0$ there exists a Lipschitz
  function $\tilde \psi$ on $\Omega$ such that $\partial E(\tilde \psi) \neq \emptyset$ and
  \begin{align*}
    \sign \psi(x) \leq \sign \tilde \psi(x) \leq \sup_{\substack{y\in \Omega\\|y - x| \leq \rho}} \sign \psi(y), \qquad x
    \in \Omega.
  \end{align*}

\end{theorem}

Note that any Lipschitz function $\psi$ with compact zero set naturally corresponds to a facet
\begin{align*}
  (\facet, \chi) := (\set{x: \psi(x) = 0}, \sign \psi).
\end{align*}
Intuitively, it is related to the surface facet (flat part of the surface) of the $(n+1)$-dimensional solid body given by the epigraph of
$\psi$, $\operatorname{epi} \psi:= \set{(x, x_{n+1}) \in \R^{n+1}: x_{n+1} \geq \psi(x)}$. If $\chi
\geq 0$ or $\chi \leq 0$, the solid body is respectively convex or concave in a neighborhood of the
surface facet.

\subsection{Literature overview}
Crystalline mean curvature was introduced independently by Angenent and Gurtin \cite{AG89} and
Taylor \cite{T91} to model the growth of small crystals. It is the first variation of
the surface energy given by a \emph{crystalline} anisotropy with respect to the change of volume.
A crystalline anisotropy refers to an anisotropic surface energy density $\sigma$ whose unit ball is a
convex polytope. When extended positively one-homogeneously to the full space, the anisotropy
$\sigma$ is a convex piece-wise linear function.
Due to the singularity of the surface energy density, the crystalline mean curvature is a nonlocal
quantity on the flat parts, or facets, of the crystal surface. As with smooth anisotropic curvatures, the crystalline mean
curvature can be evaluated as the surface divergence of a so-called Cahn-Hoffman vector field on the
surface. A Cahn-Hoffman vector field is a selection of the subdifferential of the anisotropy
evaluated at the outer normal vector on a given surface whose surface divergence has a certain
regularity. However, for crystalline anisotropies, even smooth surfaces might not have a
Cahn-Hoffman vector field whose surface divergence is a function. Therefore there has been a lot
effort to characterize sets with surfaces that admit reasonable Cahn-Hoffman vector fields. In
particular the various notions of $\phi$-regular sets were introduced ($\phi$ being the polar of
the surface energy density $\sigma$)
\cite{BN00,BNP99,BNP01a,BNP01b,B10} and the related notion of $r B_\phi$-condition \cite{BCCN09},
$B_\phi$ being the Wulff shape $\Wulff$.

A crystalline mean curvature flow or a motion by the crystalline curvature is an evolution of sets
such that the normal velocity of the surface is proportional to the crystalline mean curvature.
Already the standard mean curvature flow is known to develop singularities even when starting from
smooth initial data, and therefore a weak notion of solutions is necessary. We mention the varifold
solutions initiated by K.~Brakke \cite{B78,Il93,TT}, that however apply only to the isotropic
mean curvature flow, and the level set method approach \cite{OS,CGG,ES} that can be generalized to
the anisotropic mean curvature with a smooth anisotropy as already done in \cite{CGG}. The extension
to a crystalline anisotropy is not straightforward even in the case of a curve evolution because of
the non-local nature of the crystalline curvature \cite{GG01}. For a more detailed overview of the
literature see \cite{GP16}.

Most of the attempts at defining a reasonable notion of solutions for the crystalline mean
curvature flow have required some
kind of regularity of the evolution so that the crystalline curvature can be evaluated, such as a
$\phi$-regular flow \cite{BN00,BCCN06} or a $rB_\phi$-regular flow \cite{BCCN09}.

Recently, A.~Chambolle, M.~Morini and M.~Ponsiglione \cite{CMP} established a unique global solvability of
the flow $V = \sigma(\nu)\kappa_\sigma$ for arbitrary convex $\sigma$ and non necessarily bounded initial data, in
an arbitrary dimension.
They introduce a notion of solutions of this flow via an anisotropic sign distance function in the
spirit of H.~M.~Soner \cite{S}, but in a distributional sense that appeared in
\cite{CasellesChambolle06}, prove a comparison principle, and use the minimizing movements algorithm of
Chambolle \cite{Chambolle} to construct a solution.
This result has been recently improved in \cite{CMNP} to cover $V = \psi(\nu)(\kappa_\sigma +
f(x,t))$,
where $\psi$ is a convex anisotropy and $f$ is a given Lipschitz function.
It is not clear at this moment if their notion of solutions coincide with ours, although it is
likely. We mention the result of K.~Ishii \cite{Ishii14}, who shows that Chambolle's minimizing
movements algorithm converges to the viscosity solution of the crystalline mean curvature flow in two
dimensions.

We take a different approach using the ideas of the theory of viscosity solutions
\cite{GG98ARMA,GG01,GGP13AMSA,GGP14JMPA}. The level set
formulation of the crystalline mean curvature flow
was introduced by the authors in \cite{GP16}. Viscosity
solutions are defined via the comparison principle by testing a candidate for a solution by an
appropriate class of regular test functions. The main advantage of this approach is that it does not require the solution itself to
have any a priori regularity besides continuity. Moreover, testing a solution by a test function is
a rather local concept. Since the crystalline mean curvature might be nonlocal on flat parts of the
evolving surface, we can localize the construction of test functions to the neighborhood of such
flat parts, called (surface) facets. By choosing a local coordinate system so that the surface facet is given as a
part of a function graph where the function is equal to zero, we are at a situation covered by
Theorem~\ref{th:main}. Our notion of facet introduced in Definition~\ref{de:facet} then
corresponds to the set $\facet$ where the surface facet is located in this coordinate system and
directions in which
the surface rises above ($\chi = +1$) and falls below ($\chi = -1$) the surface facet.
The anisotropy $\SW_{\hat p}$ introduced in \eqref{sliced-W} captures the local lower-dimensional
behavior of the full anisotropy $W$
in the direction of the surface facet. A Cahn-Hoffman vector field for $(\facet, \chi)$ then corresponds
to a Cahn-Hoffman vector field on the surface in a neighborhood of the surface facet. In general,
we consider surface facets to be all the flat features of a surface with various dimensions,
including edges ($n =1$) and (planar) facets ($n = 2$), etc. Their dimension then guides the choice of the ambient
space for the facet $(\facet, \chi)$. The approximation result Theorem~\ref{th:main} and its corollary
in Theorem~\ref{th:approximate_pair} allow us to construct a large family of test functions at any
facet of a crystal. This then yields the comparison principle for viscosity solutions and
the well-posedness of the level set formulation of the crystalline mean curvature flow \cite{GP16}.
In \cite{GP16}, we showed Theorem~\ref{th:approximate_pair} by a direct construction for $n = 1,2$
and piece-wise linear anisotropies $W$ and therefore we could deduce the well-posedness of the
crystalline curvature flow in dimensions $n \leq 3$. The generalization in this paper then
automatically extends the results of \cite{GP16} to any dimension. We can also slightly simplify
the definition of viscosity solutions, see Definition~\ref{def:visc-solution}.
Let us mention that in \cite{GP16} we described a facet by a pair of open sets $(A_-, A_+)$. They
are related to the notion of facet $(\facet, \chi)$ from Definition~\ref{de:facet} as $A_\pm = \set{x:
\chi(x) = \pm 1}$.

\subsection*{Outline}
The construction of the facet $(G,\tilde \chi)$ for Theorem~\ref{th:main} will be performed first by reducing the situation to the
case of a simple set in Section~\ref{sec:reduction} using the gradient flow of the total variation
energy on a torus $\Tn$, and then combining it to produce a facet
in Section~\ref{sec:construction}. In Section~\ref{sec:level-set} we outline an important
application of Theorem~\ref{th:main} to the theory of viscosity solutions for the
level set formulation of the crystalline mean curvature problems.

\section{Approximability of a single set}
\label{sec:reduction}

We start the proof of Theorem~\ref{th:main} with a simple facet $(\facet, \chi)$ with $\chi \leq 0$.
Note that we need to construct the Cahn-Hoffman vector field only in a neighborhood of the boundary
of the approximating facet since $0 \in \Wulff = \partial W(0)$. We can thus always extend the
vector field by $0$ in the interior of the facet away from the boundary. To construct the
approximating set, we will use the gradient flow of the anisotropic total variation energy
\eqref{TV-energy} on the torus $\Tn := \Rn / \Z^n$.

Let us thus consider a single open set $D
\subset \T^n$ with a \emph{nonempty} boundary.
Let $d: \Tn \to \R$ be the signed distance function to $\partial D$ (in the torus topology of
$\Tn$) induced by $W$ with $d > 0$ in $D$.
Recall that
\begin{align}
  \label{distance}
  d(x) := \inf_{y \in D^c} W^\circ(x - y) - \inf_{y \in D} W^\circ(y - x) =
  \begin{cases}
    \inf_{y \in D^c} W^\circ(x - y), & x \in D,\\
    - \inf_{y \in D} W^\circ(y - x), & x \in D^c,
  \end{cases}
\end{align}
where we treat $D \subset \Rn$ as periodic in the sense that if $x \in D$ then $x + \Z \subset D$,
and $W^\circ$ is the polar of $W$ defined in \eqref{polar}.
Recall that $W^\circ$ is again an anisotropy, see \cite{Rockafellar}, and therefore also
Lipschitz continuous. In particular, $d$ is a Lipschitz continuous function on $\Tn$.

Fix $\delta > 0$ such that
\begin{align}
  \label{choice-of-delta}
  \set{x \in \Rn: W^\circ(x) \leq 4\delta} = 4\delta \Wulff \subset (-\frac12, \frac12)^m.
\end{align}
We consider the gradient flow (or differential inclusion) on $\Tn$
\begin{align}
  \label{gradient_flow}
  \left\{
    \begin{aligned}
      \td{\varphi}t (t) &\in - \partial E(\varphi(t)), && t > 0,\\
      \varphi(0) &= \max(-\delta, \min(d, \delta)),
    \end{aligned}
  \right.
\end{align}
where the subdifferential $\partial E$ was defined in \eqref{E-subdiff} with $\Omega = \Tn$.

It is well-known \cite{Br71,Br73} that a unique solution $\varphi \in C([0,T]; L^2(\Tn))$ exists (for any $T >
0$) and it is right differentiable for all $t \in (0, T)$.
Moreover $\frac{d^+\varphi}{dt}(t) = - \partial^0 E(\varphi(t))$, where $\partial^0 E(\psi)$
is called the minimal section (also the canonical restriction) of $\partial E(\psi)$, and it is the
element of $\partial E(\psi)$ with the smallest $L^2$-norm. In particular, $\varphi(t) \in
\domain(\partial E)$ for all $t > 0$.
Such well-posedness goes back to the work of
Y. Komura \cite{Komura67}, where initial data is assumed to be in $\domain(\partial E)$.
Finally, $\varphi(t)$ is Lipschitz for all $t > 0$ since $\varphi(0)$
is Lipschitz and \eqref{gradient_flow} has a comparison principle and is translationally invariant.
The comparison principle (and Lipschitz continuity of $\varphi$ in space) can be established, for instance, by
approximating the energy $E$ in the Mosco sense by uniformly elliptic energies, for which
\eqref{gradient_flow} is just a uniformly parabolic PDE. Mosco convergence then implies the
convergence of the resolvent problems, see \cite{GP16} for details.
Alternatively, multiply the difference of equations \eqref{gradient_flow} for two solutions $u$ and
$v$ by the positive part of the difference, $(u - v)_+ := \max(u - v, 0)$, and integrate by parts.

We want to prove that
\begin{align*}
  d_H(\partial\set{x :\pm\varphi(x, t) > 0}, \partial D) \to 0 \qquad \text{as $t \to 0+$,}
\end{align*}
where $d_H$ is the Hausdorff distance with respect to the usual Euclidean metric.

To show this, we will compare $\varphi$ with barriers of the form
\begin{align*}
  \psi_c(x) := \min \set{\min(\max(W^\circ(x - k) - c, -\tfrac c2), \delta): k \in \Z^n},
\end{align*}
for $0 < c < \delta$.
These will control the expansion of the set $\set{x : \varphi(x, t) = 0}$, which is equal to
$\partial D$ at $t = 0$.

We claim that $\psi_c(x) + Mt$ is a supersolution of \eqref{gradient_flow} for $M = M(c) > 0$ large enough. To see
this, introduce the
vector field for $x \in (-\frac 12, \frac 12)^n$ as
\begin{align*}
  z(x) :=
  \begin{cases}
    \frac {2x}c, & W^\circ(x) < \tfrac c2,\\
    \frac x{W^\circ(x)}, & \tfrac c2 \leq W^\circ(x) < c + 2\delta,\\
    \frac x{W^\circ(x)} \max(1+ \tfrac{c + 2\delta - W^\circ(x)}\delta, 0), & \text{otherwise},
  \end{cases}
\end{align*}
and then extend it periodically to $\Rn$. If $0 < c < \delta$, $z = 0$ in a
neighborhood of the boundary of $(-\frac 12, \frac 12)^n$ by \eqref{choice-of-delta}.

We claim that $z$ is a Cahn-Hoffman vector field for $\psi_c$, that is, $z \in \CH(\psi_c)$. Indeed, working only on the unit
cell $x \in (-\frac12, \frac12)^n$, note that $W^\circ(z) \leq 1$, with equality if and only if
$\tfrac c2 \leq W^\circ(x) \leq c + 2\delta$. Therefore $z \in \partial W(0)$ on $\Rn$, and thus
$z(x)
\in \partial W(\nabla \psi(x)) = \partial W(0)$ for $W^\circ(x) < \tfrac c2$ and $W^\circ(x) > c +
\delta$.
Since $\frac x{W^\circ(x)} \in \partial W(p)$ if $p \in \partial W^\circ(x)$, we have that $z \in
\partial W(\nabla \psi_c(x))$ whenever $W^\circ$ is differentiable at $x$ and $\tfrac c2 <
W^\circ(x) < c +\delta$.
As $z$ is Lipschitz, $\divo z \in L^\infty(\Tn) \subset L^2(\Tn)$ and so $z$ is a Cahn-Hoffman vector field for
$\psi_c$.

Let us set $M := \norm{\divo z}_{L^\infty}$ and $f := M - \divo z \in L^2(\Tn)$.
Observe that $(x,t) \mapsto \psi_c(x) + Mt$ is the unique solution $\phi \in C([0,\infty);
  L^2(\Tn))$ of
\begin{equation}
  \label{general_gradient_flow}
  \left\{
    \begin{aligned}
      \td{\phi}t (t) &\in - \partial E(\phi(t)) + f, && t > 0,\\
      \phi(0) &= \psi_c,
    \end{aligned}
  \right.
\end{equation}
and that $f \geq 0$, and thus it is a supersolution of \eqref{gradient_flow}.

\begin{lemma}
  \label{le:boundary-control}
  Let $\varphi$ be the unique solution of \eqref{gradient_flow}.
  For every $\rho > 0$ there exists $\tau > 0$ such that for all $x \in D$, $\dist(x, \partial D) >
  \rho$ we have $\varphi(x, t) > 0$ for all $0 \leq t < \tau$, and for all $x \in D^c$, $\dist(x,
  \partial D) > \rho$ we have $\varphi(x,t) <0$ for all $0 \leq t < \tau$. Here $\dist$ is the
  usual Euclidean distance.
\end{lemma}

\begin{proof}
Take $c \in (0, \delta)$ such that $W^\circ(x) > c$ for $\abs{x} \geq \rho$, $M = M(c)$ so that $(x,t) \mapsto
\psi_c(x) + Mt$ is a supersolution of \eqref{gradient_flow}, and $\tau = \frac{c}{2M}$. We claim
that by definition of $d$ in
\eqref{distance}, for any $x_0 \in D^c$ such that $\dist(x_0, \partial D) > \rho$, we have
\begin{align*}
  \varphi(\cdot, 0) \leq \psi_c(\cdot - x_0) \qquad \text{in $\Tn$.}
\end{align*}
To see this for one such $x_0$, we only need to show $W^\circ(x - x_0) \geq d(x) + c$ for $x \in \Rn$. If $x \in D$,
then $\abs{x - x_0} > \rho$ and hence $W^\circ(x - x_0) > c$. Set $y := \frac{c}{W^\circ(x-x_0)} (x-
x_0) + x_0$. Note that $W^\circ(y - x_0) = c$ and therefore $|y - x_0| \leq \rho$. In particular,
$y \in D^c$. Then by \eqref{distance}
\begin{align*}
  d(x) \leq
  W^\circ(x - y) = \pth{1 - \frac{c}{W^\circ(x - x_0)}} W^\circ(x - x_0) = W^\circ(x - x_0) - c.
\end{align*}
On the other hand, if $x \in D^c$, convexity and positive one-homogeneity yields for all $y \in D$
\begin{align*}
  W^\circ(y - x) + W^\circ(x - x_0)  \geq W^\circ(y - x_0) > c.
\end{align*}
Taking the infimum over $y \in D$, we get from \eqref{distance}
\begin{align*}
  W^\circ(x - x_0) - c \geq d(x).
\end{align*}

Fix thus one $x_0$ as above. By the comparison principle for \eqref{general_gradient_flow},
\begin{align*}
  \varphi(x_0, t) \leq \psi_c(0) + M t < 0, \qquad 0 \leq t < \tau.
\end{align*}
In particular, $\varphi(x_0) < 0$ for $0 \leq t < \frac{c}{2M}$.

We can argue similarly for $x_0 \in D$ using the subsolution $-\psi_c(-x) - Mt$.
\end{proof}

The lemma above allows us to control the speed of the boundary of $\set{\pm\varphi > 0}$.

\begin{corollary}
  For every $\rho > 0$ there exists $\tau > 0$ such that
  \begin{align*}
    \begin{aligned}
    d_H(\partial D, \partial \set{\varphi(\cdot, t) > 0}) \leq \rho,\\
    d_H(\partial D, \partial \set{\varphi(\cdot, t) < 0}) \leq \rho,
    \end{aligned}
    \qquad\qquad
    \text{for $0 \leq t < \tau$,}
  \end{align*}
  where $d_H$ is the Hausdorff distance with respect to the usual Euclidean metric.
\end{corollary}

\begin{proof}
  Given $\rho > 0$, by compactness there exists $\eta \in (0, \rho)$ such that for all $x \in \partial D$
  there exist $y \in D$, $z \in D^c$, such that $\dist(y,x) \leq \rho$, $\dist(z, x) \leq \rho$ while
  $\dist(y, \partial D) > \eta$, $\dist(z, \partial D) > \eta$. Take $\tau > 0$ from
  Lemma~\ref{le:boundary-control} for $\rho = \eta$, and fix $t \in (0, \tau)$.  Let us show the
  inequality for $\Gamma_t := \partial \set{\varphi(\cdot, t) > 0}$, the second one is analogous.

  First, let $x \in \Gamma_t$, which by continuity implies $\varphi(x, t) = 0$. By
  Lemma~\ref{le:boundary-control} we have $\dist(x, \partial D) \leq \eta < \rho$.

  Now suppose that $x \in \partial D$. By the choice of $\eta$ there exist $y \in D$, $z \in D^c$
  with the properties above. By Lemma~\ref{le:boundary-control} we have $\varphi(y, t) > 0$ and
  $\varphi(z, t) < 0$ for $0 \leq t < \tau$. The convexity of the norm yields $\dist(x, \partial \Gamma_t) \leq \rho$
  since the line segment connecting $y$ with $z$ must contain a point in $\partial \Gamma_t$.
\end{proof}

\section{Construction of a Cahn-Hoffman facet}
\label{sec:construction}

To prove Theorem~\ref{th:main}, let us fix a facet $(\facet, \chi)$ such that $\facet \subset (-\frac14,
\frac14)^n$. Indeed, an arbitrary facet can be reduced to this case by simple scaling since the anisotropy
is positively one-homogeneous.
We can also assume that $\rho < \frac14$. We will use the result of Section~\ref{sec:reduction} to construct the approximating
Cahn-Hoffman facet.
By continuity of $\chi$ on $\facet^c$, $\chi$ is a constant on $\Rn \setminus (-\frac 14,
\frac 14)^n$. Let us assume that $\chi = -1$ on this set. The other case can be handled
similarly.
Defining the periodic function $\overset{\circ}{\chi}(x) := \min \set{\chi(y) : y \in x + \Z^n}$, we see
that $\overset{\circ}\chi = \chi$ on $(-\frac 12, \frac12)^n$.
Set $H_+ = \{x \in \Rn: \sup_{|y - x| < \frac \rho4}\overset{\circ}\chi(y) =  +1\}$ and $H_-
= \{x \in \Rn: \sup_{|y - x| < \frac {3\rho}4}\overset{\circ}\chi(y) \geq 0\}$. Sets $H_-$ and $H_+$ are open
subsets of $\Tn$. Moreover $H_+ \subset H_-$ and $\dist(H_+, H_-^c) = \dist(\partial H_-, \partial H_+) \geq \frac \rho2$.

Define the open cube $U = (-\frac 12, \frac12)^n \subset \Rn$.
For one fixed $\delta > 0$ as in Section~\ref{sec:reduction},
let $\varphi^+$ be the solution of \eqref{gradient_flow} for set $D = H_+$, and
$\varphi^-$ be the solution $D = H_-$.
If $\partial H_+$ is empty, that is, when $H_+$ is empty, we set $\varphi^+ \equiv 0$, $z^+ \equiv
0$ below. Similarly if $\partial H_-$ is empty, implying $H_- = \Tn$, we set $\varphi^- \equiv -1$,
$z^- \equiv 0$.
Applying Lemma~\ref{le:boundary-control} with $\rho = \frac \rho8$, we have $\tau > 0$ such that
\begin{align}
  \label{phi_lvlset_bound}
  \begin{aligned}
  \{\overset{\circ}\chi = 1\} &\subset \set{\varphi^+(\cdot, t) > 0} \subset \{x: \sup_{|y - x|\leq
  \rho} \overset{\circ}\chi(y) = 1\},\\
  \{\overset{\circ}\chi \geq 0\} &\subset \set{\varphi^-(\cdot, t) \geq 0} = \{x: \sup_{|y - x|\leq
  \rho} \overset{\circ}\chi(y) \geq 0\},
  \end{aligned}
  \qquad
  0\leq t < \tau,
\end{align}
and
\begin{align}
  \label{lvlset_nonintersect}
  \dist(\set{\varphi^+(\cdot, t) > 0}, \set{\varphi^-(\cdot, t) < 0}) \geq \frac\rho4, \qquad 0
  \leq t < \tau,
\end{align}
and finally
\begin{align}
  \label{neg-on-cube-boundary}
  \varphi^-(\cdot, t) < 0 \text{ on } \partial U, \qquad 0 \leq t < \tau.
\end{align}

Let us fix $t \in (0, \tau)$ and set
\begin{align*}
  G_- := \set{x \in U: \varphi^-(x, t) < 0} \cup U^c, \qquad G_+ := \set{x \in U: \varphi^+(x,t) > 0}.
\end{align*}
By the continuity of $\varphi^\pm$ and \eqref{neg-on-cube-boundary}, $G_\pm \subset \Rn$ are open.

We define a facet by setting $G = G_-^c \cap G_+^c \subset U$, $\tilde\chi = +1$ on $G_+$ and $\tilde\chi = -1$ on $G_-$.
Then $(G, \tilde\chi)$ is a facet in the sense of Definition~\ref{de:facet} and $\chi(x) \leq \tilde\chi(x) \leq \sup_{|y - x|\leq \rho}
\chi(y)$.

To finish the proof of Theorem~\ref{th:main}, we need to show that it is actually a
$W^\circ$-$(L^2)$ Cahn-Hoffman facet by finding a Cahn-Hoffman vector field on its neighborhood $U$.
Since $\varphi_\pm$ are solutions of \eqref{gradient_flow} and for $t > 0$ we have that $\varphi^\pm(\cdot, t) \in \domain(\partial
E)$ are Lipschitz, there exist Cahn-Hoffman vector fields
$z^\pm \in X_2(\Tn)$ for $\varphi^\pm(\cdot, t)$, respectively.
We define the function $\psi: U \to \R$ as
\begin{align*}
  \psi(x) :=
  (\varphi^+)_+(x,t) - (\varphi^-)_-(x,t) =
  \begin{cases}
    \varphi_+(x, t), & x \in G_+,\\
    0, & x \in G,\\
    \varphi_-(x, t), & x \in G_- \cap U,\\
  \end{cases}
\end{align*}
where $(s)_+ := \max(s, 0)$ and $(s)_- := \max(-s, 0)$ denote the positive and negative parts,
respectively.
The functions $\varphi^\pm$ are Lipschitz continuous, and so by \eqref{lvlset_nonintersect} $\psi$
is well-defined and Lipschitz continuous function on $U$.

Set $\eta = \dist(G_-, G_+) / 4$.
Let us introduce the vector field
\begin{align*}
  z(x) := \xi^+(x) z^+(x) + \xi^-(x) z^-(x),
\end{align*}
where $\xi^\pm \in C^\infty(\Rn)$ are cut-off functions such that $\xi^\pm(x) = 1$ when $\dist(x,
G_\pm) \leq \eta$, $\xi^\pm(x) = 0$ when $\dist(x, G_\pm) \geq 2\eta$ and $0 \leq \xi^\pm \leq 1$
otherwise.
Clearly $z \in X_2(U)$.
Note that $\xi^- +\xi^+ \leq 1$.
Since $\psi \equiv 0$
on $G = \set{\psi = 0}$, we see that $z \in \partial W(\nabla \psi)$ on $U$ as $\partial W(p)$ is a
convex set. Therefore $z$
is a Cahn-Hoffman vector field on $U$, and $(G, \tilde \chi)$ is a $W^\circ$-$(L^2)$ Cahn-Hoffman facet.
The proof of Theorem~\ref{th:main} is complete.

\section{Level set crystalline mean curvature flow}
\label{sec:level-set}

In this section, we explain the application of the approximation result of Theorem~\ref{th:main}
to the theory of viscosity solutions of the crystalline mean
curvature flow problems. Specifically, let us consider the initial value problem \eqref{level-set},
\begin{align}
  \label{lse}
  \left\{
  \begin{aligned}
  u_t + F(\nabla u, \divo (\nabla W)(\nabla u)) &= 0 && \text{in $\Rn \times (0, \infty)$},\\
  \at{u}{t=0} &= u_0, && \text{in $\Rn$.}
  \end{aligned}
  \right.
\end{align}
The anisotropy $W: \Rn \to \R$ is now assumed to be piece-wise linear.
In convex analysis, convex piece-wise linear functions are also known as polyhedral functions
\cite{Rockafellar}. We call such anisotropies \emph{crystalline}. The nonlinearity $F \in C(\Rn \times \R)$ is
assumed to be nonincreasing in the second variable,
\begin{align*}
  F(p, \xi) \leq F(p, \eta), \qquad p \in \Rn, \xi \geq \eta.
\end{align*}
Thanks to this assumption, the problem \eqref{lse} has a comparison principle structure.

Problem \eqref{lse} appears as the level set formulation of an anisotropic mean curvature flow,
specifically the crystalline mean curvature flow \cite{GP16}. It can be also thought of as an anisotropic total variation flow of
non-divergence form. Of particular interest is the singular operator $\divo (\nabla W)(\nabla u)$,
which is interpreted as the minimal section (also canonical restriction) $-\partial^0 E(u)$ of the subdifferential
$\partial E(u)$ defined in \eqref{E-subdiff}. Since the problem has a comparison principle
structure, it falls within the scope of the theory of viscosity solutions. However, the extension
of the theory to problems like \eqref{lse} is quite nontrivial. In \cite{GP16}, the authors
succeeded in defining a reasonable notion of viscosity solutions for \eqref{lse}. However, due to
the difficulty of construction test functions like those in Theorem~\ref{th:approximate_pair}
in dimensions $n \geq 3$,
the well-posedness was limited to dimensions $n \leq 3$. Theorem~\ref{th:main} now
provides a sufficiently large class of test functions and thus the results of \cite{GP16} apply to
an arbitrary dimension. We will outline this in the rest of this section.

\subsection{Crystalline curvature}

Let us review the notion of the crystalline curvature of a facet. Recall the definition \eqref{ch} of the set
$\CH(\psi)$
of Cahn-Hoffman vector fields for a given Lipschitz continuous function on an open set $U \subset
\Rn$. The set of all divergences of such vector fields $\divo \CH(\psi) := \set{\divo z: z \in
\CH(\psi)}$ is a closed convex subset of $L^2(U)$. If it is nonempty, there exists a unique
element with the minimal $L^2$ norm. We denote it $\Lambda[\psi]$,
\begin{align*}
  \Lambda[\psi] := \divo z_{min}, \qquad \norm{\divo z_{min}}_{L^2(U)} = \inf \set{\norm{\divo
  z}_{L^2(U)}: z \in \CH(\psi)}.
\end{align*}
In \cite{GP16}, we proved the following comparison principle for the quantity $\Lambda$.

\begin{proposition}[{\cite[Proposition~4.12]{GP16}}]
  \label{pr:Lambda-comparison}
  If $\psi_1, \psi_2$ are two
  Lipschitz functions on an open set $U$ such their zero sets are compact subsets of $U \subset \Rn$, and
  $\Lambda[\psi_i]$, $i = 1,2$, are well-defined, then
  \begin{align*}
    \sign \psi_1 \leq \sign \psi_2 \quad \text{on $U$}
  \end{align*}
  implies
  \begin{align*}
    \Lambda[\psi_1] \leq \Lambda[\psi_2], \qquad \text{a. e. on } \set{\psi_1 =
    0} \cap \set{\psi_2 = 0}.
  \end{align*}
\end{proposition}

This allows us to introduce the $W^\circ$-$(L^2)$ minimal divergence of a $W^\circ$-$(L^2)$
Cahn-Hoffman facet.

\begin{definition}
  \label{de:minimal-divergence-facet}
Given a $W^\circ$-$(L^2)$ Cahn-Hoffman facet $(\facet, \chi)$, that is, a facet for which there exists
an open set $U \supset \facet$, Lipschitz function $\psi$ on $U$ with $\sign \psi = \chi$, and a
Cahn-Hoffman vector field $z \in \CH(\psi)$, we define the \emph{$W^\circ$-$(L^2)$ minimal
divergence of the facet $(\facet, \chi)$}, $\Lambda[(\facet, \chi)] \in L^2(\facet)$, as
\begin{align*}
  \Lambda[(\facet, \chi)] := \Lambda[\psi] \qquad \text{on $\facet$}.
\end{align*}
By Proposition~\ref{pr:Lambda-comparison}, this definition is independent of the choice of $U$ and
$\psi$.
\end{definition}

Note that Proposition~\ref{pr:Lambda-comparison} implies a comparison for $\Lambda$ of facets: if
$(F_1, \chi_1)$, $(F_2, \chi_2)$ are two $W^\circ$-$(L^2)$ Cahn-Hoffman facets that are ordered,
$(F_1, \chi_1) \preceq (F_2, \chi_2)$ in the sense of $\chi_1 \leq \chi_2$, we have
\begin{align*}
  \Lambda[(F_1, \chi_1)] \leq \Lambda[(F_2, \chi_2)] \quad \text{a.e. on $F_1 \cap F_2$}.
\end{align*}

\begin{example}
  Let $n = 1$ and $W(p) = |p|$. Then the facet $(\facet, \chi)$ with $\facet = \set{a}$, $\chi \in C(\R \setminus
  \set{a}, \set{-1,1})$ is always Cahn-Hoffman if and only if $\chi(a-) \neq \chi(a+)$.

  Facet $(\facet, \chi)$ with $\facet = [a, b]$, $a < b$, $\chi \in C(\facet^c, \set{-1,1})$ is Cahn-Hoffman. The minimal divergence is
  constant on the facet, $\Lambda[(\facet,  \chi)] = \frac{\chi(a-) + \chi(b+)}{b - a}$ on $\facet$. It is
  inversely proportional to the length of the facet.

  For $n = 0$, $\R^0 = \set{0}$ and the facet $(\set{0}, \chi)$ is always Cahn-Hoffman with
  $\Lambda[(\facet, \chi)] = 0$.
\end{example}

\begin{example}
In an arbitrary dimension for any anisotropy $W$, the rescaled Wulff shape $\facet := c\Wulff = \set{x: W^\circ(x)
\leq c}$, $\chi = 1$ on $\facet^c$ forms a $W^\circ$-$(L^2)$ Cahn-Hoffman facet $(\facet, \chi)$ for any $c > 0$ with
$\Lambda[(\facet, \chi)] = \frac nc$. This can be seen easily by taking $\psi(x) = \max(W^\circ(x) - c, 0)$,
$z(x) = \frac x{\max(W^\circ(x), c)}$.
\end{example}

\begin{example}
  For $n=2$ and $W(p) = |p_1| + |p_2|$, a rather thorough characterization of $\Lambda[(\facet, \chi)]$
  for axes-aligned polygons $\facet$ is available in \cite{LMM}. In particular, if $\Lambda[(\facet, \chi)]$
  is a constant, then $\Lambda[(\facet, \chi)] = \frac{\mathcal H^1(\partial \facet \cap \set{\chi = +1}) -
  \mathcal H^1(\partial \facet \cap \set{\chi = -1})}{|\facet|}$, where $\mathcal H^1$ is the one-dimensional
  Hausdorff measure.
\end{example}

\begin{example}
  There is an interesting relationship between $\Lambda[(\facet, \chi)]$ and the Cheeger problem,
  that is, the problem of finding the subset $\Omega \subset A$ that minimizes the ratio
  $\frac{P(\Omega)}{|\Omega|}$ among all subsets, where $P(\Omega)$ is the perimeter of $\Omega$. In
  \cite{BNP01IFB} it was shown that a facet $(\facet, \chi)$ with $\facet \subset \R^2$ convex and
  $\chi = +1$ on $\facet^c$
  has $\Lambda[(\facet, \chi)]$ constant on $\facet$ if and only if $\facet$ is a solution of the Cheeger
  problem on $\facet$ with the $W^\circ$-perimeter.
\end{example}

\subsection{Review of the notion of viscosity solutions}
\label{sec:visc-sol}

Viscosity solutions are defined as continuous functions that satisfy the comparison (maximum)
principle with a class of sufficiently regular test functions. This way we do not need to assume
any further regularity about the candidate for a solution, but we have to choose a class of test
functions that is sufficiently large. As was pointed out in the introduction, the operator
$\divo (\nabla W)(\nabla \cdot)$ might not be well-defined even for smooth functions. We therefore
restrict the family of test functions to only the stratified admissible test functions defined below.

Before the definition, let us recall convenient coordinates for the space $\Rn$, introduced in
\cite{GP16}. For a fixed
$\hat p \in \Rn$, we define $Z_1 \subset \Rn$ to be the subspace parallel to the \emph{affine hull}
$\aff \partial W(\hat p)$ of $\partial W(\hat p)$, that is, the smallest subspace such that $Z_1 +
\xi \supset \partial W(\hat p)$ for some $\xi \in \Rn$.  Its dimension is then the dimension of the subdifferential $\partial W(\hat
p)$, $k = \dim \partial W(\hat p) := \dim Z_1$.  Let $Z_2 := Z_1^\perp$ be the orthogonal subspace. Then
$\Rn = Z_1 \oplus Z_2$, $Z_1$ is isometrically-isomorphic to $\R^{k}$ and $Z_2$ is
isometrically-isomorphic to $\R^{n-k}$. Fixing such isometries $\TT_1: \R^k \to Z_1$, $\TT_2: \R^{n-k}
\to Z_2$, we can write every $x \in \Rn$ uniquely in terms of $(x', x'')$, $x = \TT_1 x' + \TT_2 x''$,
where $x' \in \R^k$ and $x'' \in \R^{n-k}$. If $k = 0$ or $k = n$, we simply take $x = x''$ or $x =
x'$, respectively.

We also introduce the convex, positively one-homogeneous function $\SW_{\hat p}:\R^k \to \R$ as
\begin{align}
  \label{sliced-W}
  \SW_{\hat p}(w) := \lim_{\lambda \to 0+} \frac{W(\hat p + \lambda \TT_1 w) - W(\hat
  p)}{\lambda}, \qquad w \in \R^k.
\end{align}
This function represents the infinitesimal structure of $W$ near $\hat p$, sliced in
the direction of $Z_1$.
Analogously to \eqref{TV-energy}, we define the functional $\SE_{\hat p}: L^2(\Omega) \to \R$.

With $\hat p$, $k$ as above, we say that a function $\psi \in Lip(\R^k)$ is a \emph{$\hat
p$-admissible support function} if $(\set{\psi = 0}, \sign \psi)$ is a $(\SW_{\hat
p})^\circ$-$(L^2)$ Cahn-Hoffman facet.
For a $\hat p$-admissible support function $\psi$, we define the nonlocal curvature-like operator
\begin{align*}
  \Lambda_{\hat p}[\psi](x) := \Lambda[(\set{\psi = 0}, \sign \psi)], \qquad x
  \in \set{\psi = 0},
\end{align*}
where $\Lambda$, defined in Definition~\ref{de:minimal-divergence-facet}, is used with the
anisotropy $W = \SW_{\hat p}$.

\begin{definition}
  \label{de:admissible-stratified-test-func}
Let $\hat p \in \Rn$, $(\hat x, \hat t) \in \Rn \times \R$, $k := \dim \partial W(\hat p)$.
We say that $\varphi(x,t) = \psi(x') + f(x'') + \hat p \cdot x + g(t)$ is an \emph{admissible
stratified faceted test function at $(\hat x, \hat t)$ with slope $\hat p$} if $f \in C^1(\R^{n -
k})$, $\nabla f(\hat x'') = 0$, $g \in C^1(\R)$, and $\psi \in Lip(\R^k)$ is a $\hat p$-admissible
support function with $\hat x' \in \interior \set{\psi = 0}$.
Note that if $k = 0$, we have $\varphi(x,t) = f(x) + g(t)$ for some $f \in C^1(\Rn)$, $g \in
C^1(\R)$.
\end{definition}

\begin{definition}[{Viscosity solution, cf. \cite[Definition~5.2]{GP16}}]
  \label{def:visc-solution}
  We say that an upper semicontinuous function $u$ is a \emph{viscosity subsolution} of $u_t +
  F(\nabla u, \divo (\nabla W)(\nabla u)) = 0$ on $\Rn \times (0,
  T)$, $T > 0$, if for any $\hat p
  \in \Rn$, $\hat x \in \Rn$, $\hat t \in (0, T)$ and any admissible
  stratified faceted test function $\varphi$ at $(\hat x, \hat t)$ with slope $\hat p$ of the form $\varphi(x, t) = \psi(x') + f(x'') + \hat p \cdot x + g(t)$
  such that the function $u - \varphi(\cdot - h)$ has a global maximum on $\Rn \times (0, T)$ at $(\hat x, \hat t)$ for all
  sufficiently small $h' \in \R^k$ and $h'' = 0$, there exists $\delta > 0$ such that
  \begin{align}
    \label{visc_subsolution}
    g_t(\hat t) + F(\hat p, \essinf_{B_\delta(\hat x')} \Lambda_{\hat p}[\psi])  \leq 0.
  \end{align}

  \emph{Viscosity supersolutions} are defined analogously as lower semicontinuous functions, replacing a global maximum with a global
  minimum, $\essinf$ with $\esssup$, and
  reversing the inequality in \eqref{visc_subsolution}.

  A continuous function that is both a viscosity subsolution and a viscosity supersolution is
  called a \emph{viscosity solution}.
\end{definition}

\begin{remark}
  Definition~\ref{def:visc-solution} at $\hat p = 0$ is a natural extension of the definition that
  appeared in an earlier paper \cite{GGP13AMSA} for anisotropies $W$ smooth outside of the origin.
  In that case, if appropriate tests are given at $\hat p \neq 0$ where $W$ is smooth, the
  definition is equivalent to the definition of $\mathcal F$-solutions \cite{G06}.
  In fact, this equivalence can be proved along the lines of the proof of
  \cite[Proposition~2.2.8]{G06}, where one has to replace $|x-y|^4$ by an appropriate function
  whose zero set consists of the Wulff shape of $W$.

  Note that we do not need to consider the special test for ``curvature-free'' directions, $\hat p =
  0$, as in \cite{GP16} due to our ability to construct admissible stratified facet
functions in an arbitrary dimension.
\end{remark}

\subsection{Comparison principle}

One of the main results of \cite{GP16} was the comparison principle
Theorem~\ref{th:comparison-principle-intro},
 assuming Theorem~\ref{th:approximate_pair}. We shall give a brief sketch of the proof of
 Theorem~\ref{th:comparison-principle-intro} for the reader's convenience.

Due to the nonlocality of the problem, we require that $u$ and $v$ are constant outside of a
compact set to avoid technical issues with unbounded facets. However, it is an interesting question
how to handle such solutions as well as boundary conditions. This will be addressed in a future work.

To show Theorem~\ref{th:comparison-principle-intro}, we follow a variant of the standard proof by contradiction.
The problem \eqref{lse} has features of a second order problem in that similarly to Ishii's lemma that is used
to construct test functions with ordered second derivatives even for only semicontinuous solutions
in the usual argument,
we need to show the existence of ordered Cahn-Hoffman facets at a contact point.
In
this article we only give the outline of the proof, and show how to apply
Theorem~\ref{th:main}. For full details, see \cite{GP16}.

We
thus suppose that we have a viscosity subsolution $u$ and a viscosity supersolution $v$ satisfying
the hypothesis of Theorem~\ref{th:comparison-principle-intro}, but for which the conclusion does not hold,
that is,
\begin{align*}
  m_0 := \sup_Q (u - v) > 0,
\end{align*}
with $Q := \Rn \times (0, T)$.

Then we follow the standard doubling-of-variables argument with an extra parameter.
That is, for $\e > 0$, $\zeta \in \Rn$, we study the maxima of the functions
\begin{align*}
  \Phi_{\zeta, \e}(x, t, y, s) := u(x,t) - v(y,s) - \frac{\abs{x-  y - \zeta}^2}{2\e} - S_\e(t,s),
\end{align*}
over $\cl Q \times \cl Q = \Rn \times [0, T] \times \Rn \times [0, T]$, where
\begin{align*}
  S_\e(t, s) := \frac{\abs{t - s}^2}{2\e} + \frac{\e}{T - t} + \frac{\e}{T - t}.
\end{align*}
The introduction of $\zeta$ helps ``flatten'' the profile of $u$ and $v$ near the point of maximum
of $\Phi_{\zeta, \e}$ that we are interested in, see Corollary~\ref{co:contact-ordering} below.

We have the following important observation.

\begin{proposition}[{cf. \cite{GG98ARMA}}]
  \label{pr:max-interior}
  There exists $\e_0 > 0$ such that for all $\e \in (0, \e_0)$, $\abs{\zeta} \leq \kappa =
  \kappa(\e)$, $\Phi_{\zeta, \e}$ does not attain its maximum on the
  boundary of $Q \times Q$.
\end{proposition}

Therefore, in what follows, we \emph{fix} one such $\e > 0$ from Proposition~\ref{pr:max-interior} and $\kappa =
\kappa(\e)$. We then write $\Phi_{\zeta} = \Phi_{\zeta, \e}$ and $S = S_\e$.

Now we follow the notation from \cite{GG98ARMA,GP16}. We define the maximum of $\Phi_\zeta$ as
\begin{align*}
  \ell(\zeta) := \max_{\cl Q \times \cl Q} \Phi_\zeta,
\end{align*}
and the set of the points of maximum
\begin{align*}
  \mathcal A(\zeta) := \argmax_{\cl Q \times \cl Q} \Phi_\zeta : =\set{(x,t,y,s) \in \cl Q \times
  \cl Q: \Phi_\zeta(x,t,y,s) = \ell(\zeta)}.
\end{align*}
Moreover, the set of gradients at maximum will be denoted as
\begin{align*}
  \mathcal B(\zeta) := \set{\frac{x - y - \zeta}\e :(x,t,y,s) \in \mathcal A(\zeta)}.
\end{align*}

We have the following compactness property.

\begin{proposition}[{cf. \cite[Proposition~7.3]{GP16}}]
  The graphs of $\mathcal A(\zeta)$ and $\mathcal B(\zeta)$ over $\abs{\zeta} \leq \kappa$ are
  compact.
\end{proposition}

This compactness and the simple structure of $\partial W$, piece-wise constant on relatively open
convex sets, allow us to use the Baire category theorem to find a direction in which $u$ and $v$
have certain flatness.

\begin{proposition}[{cf. \cite[Proposition~7.4]{GP16}}]
  \label{pr:Xi-existence}
  There exists a set $\Xi \in \Rn$, a vector $\hat \zeta \in \Rn$ and
  $\lambda > 0$ such that $|\hat \zeta| + 2\lambda < \kappa$, $\partial W(p)$ is independent of $p
  \in \Xi$, and
  \begin{align*}
    \mathcal B(\zeta) \cap \Xi \neq \emptyset \qquad \text{for all $|\zeta - \hat \zeta| < 2\lambda$.}
  \end{align*}
  The set $\Xi$ can be taken as a relatively open convex set in its affine hull $\aff \Xi$, which is
  orthogonal to $\aff \partial W(p)$, $p \in \Xi$.
\end{proposition}

In other words, the proposition implies the existence of a set $\Xi \subset \Rn$ and a point $\hat
\zeta\in \Rn$
such that for any $\zeta \in \Rn$ close to $\hat \zeta$ there
exists a point of maximum $(\hat x, \hat t, \hat y, \hat s) \in \mathcal A(\zeta)$ of $\Phi_\zeta$
such that $\frac{\hat x - \hat y - \zeta}\e \in \Xi$. Noting that $\Phi_\zeta(x, t, \hat y, \hat s)
\leq \Phi_\zeta(\hat x, \hat t, \hat y, \hat s)$,
we see that $\frac{\hat x - \hat y - \zeta}\e$ is the
gradient of the smooth test function $(x,t) \mapsto \frac{\abs{x - \hat y - \zeta}^2}{2\e} + S(t, \hat s)$
touching $u$ at $(\hat x, \hat t)$. A similar reasoning applies for $v$. Since this gradient always
falls into $\Xi$, we recover some flatness of $u$ and $v$ in the directions orthogonal to $\aff
\Xi$, in the sense of the following lemma.

\begin{lemma}[{cf. \cite[Lemma~7.6]{GP16}}]
\label{le:max-constancy}
Suppose that there exist $\hat p, \hat \zeta \in \Rn$, a subspace $Z_2 \subset \Rn$ and $\lambda > 0$
such that $|\hat \zeta| + 2\lambda < \kappa$
and
\begin{align*}
\mathcal B(\zeta) \cap (\hat p + Z_2) \neq \emptyset  \qquad \text{for all $|\zeta - \hat \zeta| < 2\lambda$.}
\end{align*}
Then
\begin{align*}
\ell(\zeta) - \hat p \cdot \zeta \equiv const
\quad \text{for all } \zeta \in \hat \zeta + Z_1,\ |\zeta - \hat \zeta| < 2\lambda,
\end{align*}
where $Z_1 := Z_2^\perp$.
\end{lemma}

The following can be derived about the behavior of $u$ and $v$ from the above lemma. Recall that
$x' \in \R^k$ is identified with the orthogonal projection of $x$ onto $Z_1$ and $x'' \in \R^{n-k}$ with the
projection onto $Z_2$.
\begin{corollary}[{cf. \cite[Corollary~7.7]{GP16}}]
\label{co:contact-ordering}
Suppose that we have $\hat p$, $\hat \zeta$, $\lambda$, $Z_1$ and $Z_2$
as in Lemma~\ref{le:max-constancy}.
Define
\begin{align*}
\theta(x,t,y,s) := u(x,t) - v(y,s) - \frac{|x'' - y''- \hat \zeta''|^2}{2\e}
- \hat p' \cdot (x' - y'-\hat \zeta') - S(t,s).
\end{align*}
Then for any $(\hat x,\hat t,\hat y,\hat s) \in \mathcal A(\hat \zeta)$
such that $\frac{\hat x' - \hat y' - \hat \zeta'}{\e} = \hat p'$
we have
\begin{align*}
\theta(x,t,y,s) \leq \theta(\hat x, \hat t, \hat y,\hat s) \quad \text{for } (x,t),(y,s) \in \cl Q,
\ \abs{x' - y' - (\hat x' - \hat y')} \leq \lambda.
\end{align*}
\end{corollary}

Now we want to construct admissible stratified faceted test functions for $u$ and $v$ to reach a
contradiction with the definition of a viscosity solution. Therefore for the rest of the proof, we
fix $\Xi$, $\hat \zeta$ and $\lambda$ from Proposition~\ref{pr:Xi-existence}, and a point
of maximum $(\hat x,\hat t,\hat y,\hat s) \in \mathcal A(\hat \zeta)$ such that $\hat p := \frac{\hat x -
\hat y - \hat \zeta}\e \in \Xi$.
Given that $\partial W(p)$ is independent of $p \in \Xi$, we set $Z_1 \subset \Rn$ parallel to the affine hull
$\aff\partial W(p)$, $p \in \Xi$. The convexity of $W$ implies that $\Xi - \hat p \subset Z_2 := Z_1^\perp$,
see \cite[Proposition~3.1]{GP16}, and therefore Lemma~\ref{le:max-constancy} and
Corollary~\ref{co:contact-ordering} apply.

Let us set $k := \dim \partial W(\hat p) = \dim Z_1 = n - \dim Z_2$. We need to construct two $\hat p$-admissible
support functions on $\R^k$ ordered so that the comparison principle in
Proposition~\ref{pr:Lambda-comparison} applies, and which we can use to build the admissible
faceted test functions at $(\hat x, \hat t)$ for $u$ and at $(\hat y, \hat s)$ for $v$, with slope
$\hat p$. Since this is trivial if $k = 0$, we will from now assume that $k \geq 1$.

\subsubsection{Facet construction}

We first observe that $\hat p \perp Z_1$ since $W$ is positively one-homogeneous and
therefore $\hat p' =
0$. We need to find two facets to which to apply Theorem~\ref{th:main}. We
follow \cite{GP16}. First we define the functions
\begin{align*}
\begin{aligned}
\hat u(w) &:= u(\TT_1 w + \hat x, \hat t) - u(\hat x, \hat t),\\
\hat v(w) &:= v(\TT_1 w + \hat y, \hat s) - v(\hat y, \hat s),
\end{aligned}
&&& w \in \R^k,
\end{align*}
and their level sets
\begin{align*}
\hat U := \set{w \in \R^k : \hat u(w) \geq 0}, &&&
\hat V := \set{w \in \R^k : \hat v(w) \leq 0}.
\end{align*}

\begin{lemma}
  \label{le:bounded-U-V}
  Sets $\hat U$ and $\hat V$ are closed with compact boundary, and at least one is compact.
\end{lemma}

\begin{proof}
  Closedness follows from semicontinuity of $\hat u$ and $\hat v$.
  Since $\sup_Q (u - v) > 0$ and $u \leq v$ outside of a bounded set, we have $u(\hat x, \hat t) >
  v(\hat y, \hat s)$ and therefore $\hat u < \hat v$ outside of a bounded set.
  Since $\hat u$ and $\hat v$ are also both constant outside of a bounded set, $\hat U$ and
  $\hat V$ cannot be both unbounded, and their boundaries must be compact. The lemma follows.
\end{proof}

For convenience, we set
\begin{align*}
\xi_u(x'', t) &:= \frac{|x'' - \hat y'' - \hat \zeta''|^2}{2\e}
- \frac{|\hat x'' - \hat y'' - \hat \zeta''|^2}{2\e} + S(t, \hat s) - S(\hat t, \hat s),\\
\xi_v(y'', s) &:= \frac{|\hat x'' - \hat y'' - \hat \zeta''|^2}{2\e}
- \frac{|\hat x'' - y'' - \hat \zeta''|^2}{2\e} + S(\hat t, \hat s) - S(\hat t, s),
\end{align*}
$x'', y'' \in \R^{n-k}$, $t,s \in \R$.
From Corollary~\ref{co:contact-ordering} we deduce
\begin{align}
  \label{uv-nonpos}
  \begin{aligned}
u(x,t) - u(\hat x, \hat t) - \xi_u(x'', t) &\leq 0,
&&\text{for } (x,t) \in \cl Q, \dist(x' - \hat x', \hat V) \leq \lambda,\\
v(y,s) - v(\hat y, \hat s) - \xi_v(y'', s) &\geq 0,
&&\text{for } (y,s) \in \cl Q, \dist(y' - \hat y', \hat U) \leq \lambda.
  \end{aligned}
\end{align}
For reasons that shall become apparent in the proof of Lemma~\ref{le:pair-properties} below, we set $r :=
\frac{\lambda}{5}$. From the definition of $\hat U$, $\hat V$, semicontinuity of $u$, $v$, and the fact that $u$ and
$v$ are constant outside of a bounded set,
there exists $\delta > 0$ such that
\begin{equation}
  \label{uv-neg}
  \begin{aligned}
u(x,t) - u(\hat x, \hat t)  - \xi_u(x'', t) &< 0,
&&\dist(x' - \hat x', \hat U) \geq r, |x'' - \hat x''| \leq \delta, |t - \hat t| \leq \delta,\\
v(y,s) - v(\hat y, \hat s)  - \xi_v(y'', s) &> 0,
&&\dist(y' - \hat y', \hat V) \geq r, |y'' - \hat y''| \leq \delta, |s - \hat s| \leq \delta.
  \end{aligned}
\end{equation}

We now introduce $k$-dimensional facets $(F_u, \chi_u)$ and $(F_v, \chi_v)$. $F_u$ and $\chi_u \in
C(F_u^c, \set{-1, 1})$ are defined as
\begin{align}
  \nonumber
  F_u &:= \{w \in \R^k: \dist(w, \hat U) \leq 2r \wedge (\dist(w, \hat U) \geq 2r \vee \dist(w, \hat
  V) \leq \lambda - r)\},\\
  \label{chi_u_v}
  \chi_u(w) &:=
  \begin{cases}
    1, & \dist(w, \hat U) < 2r \wedge \dist(w, \hat V) > \lambda -r,\\
    -1, & \dist(w, \hat U) > 2r,\\
    0, & \text{otherwise},
  \end{cases}
\end{align}
and $F_v$, $\chi_v \in C(F_v^c, \set{-1, 1})$ are defined analogously, swapping $\hat U$ with $\hat V$ and $F_u$
with $F_v$.
Note that $F_u$ and $F_v$ are compact due to Lemma~\ref{le:bounded-U-V}.

Applying Theorem~\ref{th:main} with $\rho = r$, for anisotropy $\SW_{\hat p}$,
we obtain a $W_{\hat p}^{{\rm sl}, \circ}$-($L^2$) Cahn-Hoffman facet $(G_u, \tilde \chi_u)$  for
the facet $(F_u, \chi_u)$. Similarly, with anisotropy $p \mapsto \SW_{\hat p}(-p)$, we obtain
a facet $(G_v, \tilde \chi_v)$ for $(F_v, \chi_v)$. Note that $(G_v, - \tilde \chi_v)$ is a
$W_{\hat p}^{{\rm sl}, \circ}$-($L^2$) Cahn-Hoffman facet.

To finish the construction of the test functions, we set for $x', y' \in \R^k$
\begin{align*}
\tilde u(x' - \hat x') &:=\sup_{|x'' - \hat x''| \leq \delta}\sup_{|t - \hat t| \leq \delta} \bra{ u(x,t) - u(\hat
x, \hat t)  - \xi_u(x'', t)},\\
\tilde v(y' - \hat y') &:= \inf_{|y'' - \hat y''| \leq \delta}\inf_{|s - \hat s| \leq \delta} \bra{v(y,s) - v(\hat
y, \hat s)  - \xi_v(y'', s)}.
\end{align*}

\begin{lemma}
\label{le:pair-properties}
The facets $(G_u, \tilde \chi_u)$ and $(G_v, -\tilde \chi_v)$ have the following properties:
\begin{enumerate}
  \item The facets are ordered, and ordered with respect to the ``facets'' of $\tilde
    u$ and $\tilde v$, namely,
\begin{align}
\label{pair-order}
\sup_{|y - x| \leq r} \sign \tilde u(y) \leq \tilde \chi_u(x) \leq - \tilde\chi_v(x) \leq
\inf_{|y-x| \leq r} \sign \tilde v(y), \qquad x \in \Rn.
\end{align}
\item The origin $0$ lies in the interior of the intersection of the facets, i.e.,
\begin{align*}
  \cl B_r(0) \subset G_u \cap G_v.
\end{align*}
\end{enumerate}
\end{lemma}

\begin{proof}
  The lemma was previously proved in \cite[Lemma~4.6]{GGP13AMSA}, using a different notation. For
  the reader's convenience, we present a self-contained proof using the new notion of facets.
  To simplify the notation, we write $\sup_{|y-x|\leq r} \chi(y)$ as $(\sup^r \chi)(x)$, and analogously for
  $\inf^r$.

  For (a), note that by construction from Theorem~\ref{th:main}, $\chi_u \leq \tilde \chi_u \leq \sup^r \chi_u$ and
  $\inf^r (-\chi_v) \leq -\tilde\chi_v \leq -\chi_v$.
  Therefore the first inequality in \eqref{pair-order} will follow from $\sup^r \sign \tilde u \leq
  \chi_u$, the second will follow from $\sup^r \chi_u \leq \inf^r (- \chi_v)$ and the third from
  $- \chi_v \leq \inf^r \sign \tilde v$.

  To show $\sup^r \chi_u \leq \inf^r (-\chi_v)$, we show the equivalent $\chi_u(w) \leq -\chi_v(z)$
  for all $|z - w| \leq 2r$. Fix therefore $w, z$ with $|z - w| \leq 2r$. When $\chi_u(w) = -1$,
  there is nothing to show since $-\chi_v \geq -1$. If $\chi_u(w) = 0$
  then by \eqref{chi_u_v} $\dist(w, \hat U) \leq 2r$, and by the triangle inequality $\dist(z, \hat U) \leq 4r \leq
  \lambda - r$ and therefore $-\chi_v(z) \geq 0 = \chi_u(w)$.
  Similarly, if $\chi_u(w) = 1$ then by \eqref{chi_u_v} $\dist(w, \hat V) > \lambda - r$, by the
  triangle inequality $\dist(z, \hat V) > \lambda - 3r \geq 2r$ and therefore $-\chi_v(z) = 1 =
  \chi_u(w)$. We conclude that $\chi_u(w) \leq \chi(z)$ for all $w, z$, $|z - w| \leq 2r$.

  To show $\sup^r \sign\tilde u \leq \chi_u$, fix $w, z$ with $|z - w| \leq r$.
  If $\tilde u(w) < 0$, then automatically $\sign\tilde u(w) = -1 \leq \chi_u(z)$.
  If $\tilde u(w) = 0$, then by \eqref{uv-neg} $\dist(w, \hat U) < r$. By the
  triangle inequality, $\dist(z, \hat U) < 2r$ and so by \eqref{chi_u_v} $\chi_u(z) \geq 0 =
  \sign \tilde u(w)$. Finally, if $\tilde u(w) > 0$, by \eqref{uv-nonpos} $\dist(w, \hat V) >
  \lambda$ and by \eqref{uv-neg} $\dist(w, \hat U) < r$. Therefore the triangle inequality implies
  $\dist(z, \hat U) < 2r$ and $\dist(z, \hat V) > \lambda - r$, and so by \eqref{chi_u_v}
  $\chi_u(z) = 1 = \sign\tilde u(w)$.
  The proof of $- \chi_v \leq \inf^r \sign \tilde v$ is analogous.

  For (b), recall that $0 \in \hat U \cap \hat V$. Therefore by the triangle inequality $\dist(w,
  \hat U) \leq r \leq 2r$ and $\dist(w, \hat V) \leq r \leq \lambda - r$ for any $|w| \leq r$. In
  particular $\chi_u(w) = 0$ by \eqref{chi_u_v}. Furthermore, for any $z$, $|z - w| \leq r$,
  $\dist(z, \hat U) \leq 2r$, $\dist(z, \hat V) \leq 2r \leq \lambda - r$ and therefore $\chi_u(z)
  =0$. We conlude that $0 = \chi_u(w) \leq \tilde \chi_u(w) \leq (\sup^r \chi_u)(w) = 0$, in
  particular, $w \in G_u$. An analogous argument implies $w \in G_v$.
\end{proof}

Let us now set $m := \sup \{\tilde u(w): \dist(w, \{\tilde \chi_u = -1\}) \leq \frac r2\}$
and define
\begin{align*}
  \psi_u(x) := \max\pth{\frac{2\sup \tilde u}r \dist(x, G_u) \tilde\chi_u(x), m}.
\end{align*}
Then due to the ordering in
Lemma~\ref{le:pair-properties}(a) and the fact that $\tilde u$ is a constant outside of a bounded
set, $m < 0$, and
$\tilde u(\cdot - w) \leq \psi_u$ for all $|w| \leq \frac r2$. Note that $\sign \psi_u =
\tilde\chi_u$. In a similar way we can construct $\psi_v$ such that
$\psi_v \leq \tilde v(\cdot - w)$ for all $|w| \leq \frac r2$.

The functions $\varphi_u(x,t) := \psi_u(x' - \hat x') + \xi_u(x'', t)$ and $\varphi_v := \psi_v(x' - \hat y')
+ \xi_v(x'', t)$ are admissible stratified faceted test functions at $(\hat x, \hat t)$ and $(\hat
y, \hat s)$, respectively, with slope $\hat p$, in the sense of
Definition~\ref{de:admissible-stratified-test-func}.
Since $u$ is bounded above and $\psi_u$ is bounded, by modifying and smoothly extending $\xi_u$ for $|x'' - \hat x''| > \frac
\delta2$, $|t - \hat t| > \frac \delta2$, we can assume that $u - \varphi_u(\cdot - h)$ has a
global maximum at $(\hat x, \hat t)$ in $Q$ for all sufficiently small $h'$, $|h'| \leq \frac r2$, and $h'' = 0$.
A similar reasoning applies to $\varphi_v$. Therefore $\varphi_u$ and $\varphi_v$ are test functions
in the sense of Definition~\ref{def:visc-solution}.

From the definition of viscosity solutions, Definition~\ref{def:visc-solution}, we infer that for
some $\tilde r \in (0, r)$,
\begin{align}
  \label{contact-ineq}
  \begin{aligned}
    (\xi_u)_t(\hat t) + F\pth{\hat p, \essinf_{B_{\tilde r}(0)} \left[ \Lambda_{\hat p}[\psi_u] \right]} &\leq 0,\\
    (\xi_v)_t(\hat s) + F\pth{\hat p, \esssup_{B_{\tilde r}(0)} \left[ \Lambda_{\hat p}[\psi_v] \right]} &\geq 0.
  \end{aligned}
\end{align}
On the other hand,
Lemma~\ref{le:pair-properties}(a--b) and the comparison principle
Proposition~\ref{pr:Lambda-comparison} imply
\begin{align}
\label{essinf sup order}
\essinf_{B_{\tilde r}(0)} \left[ \Lambda_{\hat p}[\psi_u] \right] \leq \esssup_{B_{\tilde r}(0)} \left[
\Lambda_{\hat p}[\psi_v] \right],
\end{align}
and therefore by subtracting the inequalities in \eqref{contact-ineq} and applying
the ellipticity of $F$ we obtain
\begin{align*}
  0 < \frac \e{(T - \hat t)^2} + \frac \e{(T - \hat s)^2} + F\pth{\hat p, \essinf_{B_{\tilde r}(0)} \left[
  \Lambda_{\hat p}[\psi_u] \right]} - F\pth{\hat p, \esssup_{B_{\tilde r}(0)} \left[ \Lambda_{\hat p}[\psi_v] \right]}
\leq 0,
\end{align*}
a contradiction.
This finishes the proof of the comparison principle, Theorem~\ref{th:comparison-principle-intro}.

\subsection{Stability and existence of solutions}

With the comparison principle, Theorem~\ref{th:comparison-principle-intro}, valid in any dimension, the stability with respect to approximation
by regularized problems, Theorem~\ref{th:stability-intro}, and the well-posedness,
Theorem~\ref{th:well-posedness-intro}, proved originally in
\cite{GP16} immediately generalize to all dimensions. Let us give a brief outline of the approach,
for all details see \cite{GP16}. The argument in the simplified form presented below basically first appeared
in \cite{GGP14JMPA}.

In the standard viscosity theory, the existence of solutions usually follows from Perron's method:
the largest subsolution (or the smallest supersolution) turns out to be a solution of the problem,
see
\cite{CIL,G06}. However, it is not clear whether Perron's method can be used for the crystalline
curvature problem \eqref{level-set} due to the very strong nonlocality of the curvature operator,
except in one dimension \cite{GG98ARMA,GG01} when the speed of a facet is constant. In case when
the speed of the facet is not constant, this approach seems to be difficult except in a one-dimensional setting, where Perron's method is applied to construct a graph-like solution with
non-uniform driving force term by careful classification of speed profile of each facet \cite{GGN}. To be more specific,
the standard shift of a faceted test function to create a larger subsolution when the largest subsolution
fails to be a supersolution and reach a contradiction (Ishii's shift) cannot be performed unless
the crystalline curvature $\Lambda$ is constant on the facet. Therefore we use the stability of solutions with respect to
regularization of $W$ in place of
Perron's method.

\subsubsection{Stability}
We consider two modes of regularization of $W$:
\begin{enumerate}
  \item $W_m \in C^2(\Rn)$, $a_m^{-1} \leq \nabla^2 W_m \leq a_m$ for some $a_m > 0$, and $W_m
    \searrow W$, or
  \item $W_m$ are smooth anisotropies, that is, $W_m$ is an anisotropy, $W_m \in C^2(\Rn \setminus \set0)$ and $\set{p:
    W_m(p) \leq 1}$ is strictly convex, such that $W_m \to W$ locally uniformly.
\end{enumerate}
The regularization (a) yields a sequence of degenerate parabolic problems that are within the
classical viscosity theory \cite{CIL}, while (b) regularizes the crystalline curvature operator by
smooth anisotropic curvatures studied in \cite{CGG} when $F$ comes from a level set formulation, or
in \cite{GGP14JMPA,GGP13AMSA} for general $F$.

Note that both regularizations produce \emph{local}
problems, except at $\nabla u_m = 0$ in (b) .
The main difficulty in proving the stability property with respect to the approximation is
therefore again caused by the nonlocality of the crystalline curvature operator.
In the limit $m \to \infty$, the nonlocal information contained in $\Lambda$ must be
recovered. This is achieved with the help of a variant of the perturbed test function method.

Suppose therefore that we approximate $W$ by a sequence of regular $W_m$ as in (a) above and obtain a sequence of solutions
$u_m$ of the regularized problems. We want to show that $u(x, t) = \limsup_{(y, s, m) \to
(x, t, \infty)} u_m(y,s)$ is a subsolution of \eqref{lse} by verifying
Definition~\ref{def:visc-solution}. Let $\varphi$ be an admissible stratified faceted
test function at a point $(\hat x, \hat t)$ with slope $\hat p$ satisfying the assumptions in
Definition~\ref{def:visc-solution}. We need to show \eqref{visc_subsolution}.

To simplify the explanation, let us assume that $k = n$ so that $\hat p = 0$, $Z_1 = \Rn$ and
$\varphi(x,t) = \psi(x) + g(t)$, where $\psi \in Lip(\Rn)$ is $0$-admissible support function.
By definition, $(A, \chi) := (\set{\psi = 0}, \sign \psi)$ is a $(\SW_0)^\circ$-($L^2$)
Cahn-Hoffman facet.
For the treatment of the general case, see the details in \cite{GP16}.
By adding constants and translation, we can assume that $(\hat x, \hat t) = (0, 0)$ and $u(\hat x,
\hat t) = u(0,0) = 0$, $g(\hat t) = g(0) = 0$.
By Definition~\ref{def:visc-solution}, we may assume that $u - \varphi(\cdot - h)$ has a global
maximum at $(\hat x, \hat t) = (0,0)$ for all $|h| \leq \rho$ for some $\rho > 0$, and $\psi(x) = 0
\Leftrightarrow x \in A$ for
$|x| \leq \rho$.

There are now a few difficulties with trying to follow the standard stability argument for viscosity
solutions. The first is the smoothness of $\varphi$. We need at least $C^2$-regularity in space to be
able to use $\varphi$ as a test function for the approximate problems called $m$-problems. Let us
therefore assume that $\varphi$ is in fact smooth so that this is not an issue. The second problem
arises when we try to find a subsequence $m_l \to \infty$ and points of maxima $(x_l, t_l)$ of $u_{m_l} -
\varphi$ such that $(x_l, t_l) \to (\hat x, \hat t) = (0,0)$. Such a sequence exists in general
only if $u - \varphi$ has a \emph{strict} maximum at $(0, 0)$. In the standard argument, this is
ensured by a smooth perturbation of $\varphi$, for instance by adding a term like $|x|^4 + |t|^2$
to $\varphi$.
Suppose that we have such a sequence $(x_{m_l}, t_{m_l})$ of maxima converging to $(0,0)$. Then by
the definition of the viscosity solution of the $m$-problem, we have
\begin{align*}
  [\varphi_t + F(\nabla \varphi, \divo(\nabla W_{m_l})(\nabla \varphi))](x_{m_l}, t_{m_l}) \leq 0,
  \qquad \text{for all $l$}.
\end{align*}
As $\divo(\nabla W_m)(\nabla \varphi) = \trace [(\nabla^2 W_m)(\nabla \varphi)\nabla^2 \varphi]$
and $\nabla \varphi(0, 0) = 0$, there is little hope that this local quantity will converge to anything
useful due to the singularity of $W$ at $0$, much less to $\Lambda_{\hat p}[\psi](\hat x)$, which
is nonlocal.

The key idea is to introduce a uniform perturbation of $\varphi$ that depends on $m$, so that
it captures the necessary nonlocal information. This basic scheme was introduced with great success in
the viscosity theory by Evans \cite{Evans89}, where it is called the perturbed test function
method.
In fact, this idea was actually carried out in the one-dimensional setting, where the test
function is taken essentially as $W^\circ_m$ \cite{GG1a}.
However, in higher dimensional case, one has to test with more functions, which requires a new idea
for a choice of test functions depending on $m$. Such an idea has first appeared in
\cite{GGP14JMPA}. We shall sketch it below.

As $-\Lambda_{\hat p}(\psi)$ coincides on the facet with the minimizing element of the
subdifferential of the anisotropic total variation energy at $\psi$, it can be approximated by its
resolvent problem. This problem is equivalent to performing one step of the implicit Euler
discretization of the anisotropic total variation flow. To have compactness, we modify $\psi$ far
away from them facet and rescale so that it is $\Z^n$-periodic.  For given $a > 0$, we find the
unique solutions $\psi_a, \psi_{a,m} \in L^2(\T^n)$ of the resolvent problems
\begin{align}
  \label{resolvent-problems}
  \frac{\psi_{a,m} - \psi}a &\in - \partial E_m(\psi_{a,m}), & \frac{\psi_{a  } - \psi}a &\in - \partial E(\psi_{a  }),
\end{align}
where $E$ and $E_m$ are the energies defined in \eqref{TV-energy} with $\Omega = \T^n = \Rn /
\Z^n$.
It is possible to modify $\psi$ away from the facet in such a way that $\psi \in Lip(\T^n)$ and
$\partial E(\psi) \neq \emptyset$ since $\psi$ is a $0$-admissible support function.

The resolvent problems have a number of very useful properties. Since $W_m$ are smooth, the problem
for $\psi_{a,m}$ is a quasilinear elliptic problem. Therefore $\psi_{a,m} \in C^2(\Tn)$ by the elliptic
regularity.
Moreover, due to the monotone convergence of $W_m \to W$, $E_m$ converges in Mosco sense to $E$,
which implies the resolvent convergence $\psi_{a,m} \to \psi_a$ in $L^2(\Tn)$ as $m \to \infty$
\cite{Attouch}. By the comparison principle and the translation invariance of the resolvent
problems \eqref{resolvent-problems}, $\norm{\nabla \psi_{a,m}}_\infty \leq \norm{\nabla
\psi}_\infty$. Therefore $\psi_{a,m} \to \psi_a$ uniformly and $\norm{\nabla \psi_a}_\infty \leq
\norm{\nabla \psi}_\infty$.
Finally, since $\partial E(\psi) \neq \emptyset$, $\psi_a \to \psi$ uniformly and $\frac{\psi_a - \psi}a \to - \partial^0
E(\psi)$ in $L^2(\Tn)$ as $a \to 0+$. Recall that $\Lambda_0[\psi] = -\partial^0 E(\psi)$ on the
facet of $\psi$.
This construction yields perturbed test functions $\varphi_{a,m}(x,t) = \psi_{a,m}(x) + g(t)$ and
$\varphi_a(x,t) = \psi_a(x) + g(t)$.

We now turn our attention to a compact neighborhood of the facet of $\psi$,
\begin{align*}
O := \set{x: \dist(x, A) \leq \rho}.
\end{align*}
We assume that we have modified $\psi$ above only far away from the facet $(A,
\chi)$ so that the value of $\psi(\cdot - w)$ does not change on $O$ for all $|w| \leq \rho$.
For convenience, we define
\begin{align*}
  \bar u(x) := \sup_{|t| \leq \rho} u(x, t) - g(t).
\end{align*}
Note that $\bar u - \psi(\cdot - h) \leq 0$ on $O$ for $|w| \leq \rho$, with equality at $x = 0$.

We set $\delta := \frac \rho5$ and define the critical set
\begin{align*}
  N := \set{x \in O: \bar u(x) \geq 0, \psi(x - w) \leq 0 \text{ for some } |w| \leq \delta}.
\end{align*}
We can deduce that \cite[Corollary~8.3]{GP16}
\begin{align*}
  \bar u(x)& \leq 0, \quad \psi(x - z) \geq 0, \qquad \text{for all } \dist(x, N) \leq 3 \delta, |z| \leq
  \delta,\\
  \dist(N, \partial O) &\geq 4\delta.
\end{align*}
In particular, we see that for any $|z| \leq \delta$ and any $\alpha > 0$ all maxima of $\bar u -
\alpha \psi(\cdot - z)$ in $\set{x: \dist(x, N) \leq 3\delta}$ lie in $N$.
We can therefore make the Lipschitz constant $\norm{\nabla \psi}_\infty$ arbitrarily small by
multiplying $\psi$ by small $\alpha > 0$ in the sequel.
By adding $|t|^2$ to $g(t)$ if necessary, we may assume that all maxima of $u - \varphi(\cdot - z,
\cdot)$ are located at $t = 0$.

For $a > 0$ let us fix $z_a$ such that $|z_a| \leq \delta$ and $\psi_a(z_a) = \min_{|w| \leq
\delta} \psi_a(w)$. This choice will become important later.

By the above consideration and the uniform convergence of $\psi_a \to \psi$, there exists $a_0 > 0$ such that all the maxima of $u
- \varphi_a(\cdot - z_a, \cdot)$ in $M^{3\delta}$ lie in $M^\delta$, where $M^s := \set{(x,t):
\dist(x, N) \leq s, |t| \leq s}$.

Now for every $a \in (0, a_0)$, by the uniform convergence of $\psi_{a,m} \to \psi_a$ and properties
of half-relaxed limits, there exist a point $(x_a, t_a) \in M^\delta$ of maximum of $u -
\varphi_a(\cdot - z_a, \cdot)$ and sequences $m_l \to \infty$, $(x_l, t_l) \to (x_a, t_a)$ as $l
\to \infty$, where $(x_l, t_l)$ is a point of maximum of $u_{m_l} - \varphi_{a,m_l}(x_l - z_a,
t_l)$. By the uniform Lipschitz continuity of $\psi_{a,m}$, we can assume that there exists $p_a
\in \Rn$,
$|p_a|\leq \norm{\nabla \psi}_\infty$, such that $\nabla
\psi_{a,m_l}(x_{m_l}) \to p_a$ as $l \to \infty$.
Since $\varphi_{a,m}$ are smooth, the definition of viscosity solution $u_m$ implies
\begin{align*}
  g'(t_l) + F(\nabla \psi_{a,m_l}, \trace [(\nabla^2 W_{m_l})(\nabla \psi_{a, m_l}) \nabla^2
  \psi_{a, m_l}]) (x_l - z_a) \leq 0.
\end{align*}
Since the trace operator is just $\frac{\psi_{a,m_l} - \psi}a$, the uniform convergence implies
in the limit $l \to \infty$
\begin{align*}
  g'(t_a) + F\big(p_a, \frac{\psi_a - \psi}a (x_a - z_a)\big) \leq 0.
\end{align*}
$z_a$ was chosen above in such a way that a geometric lemma, \cite[Lemma~8.5]{GP16}, implies
\begin{align*}
  \frac{\psi_a - \psi}a (x_a - z_a) \leq \min_{|w| \leq \delta} \frac{\psi_a - \psi}a(w).
\end{align*}
Monotonicity of $F$ in the second variable thus yields
\begin{align*}
  g'(t_a) + F\big(p_a, \min_{|w|\leq \delta}\frac{\psi_a - \psi}a (w)\big) \leq 0.
\end{align*}
Note that $t_a \to 0$ as $a \to 0+$.
By finding a sequence of $a_j \to 0+$ such that there exist $p \in \Rn$, $|p| \leq \norm{\nabla
\psi}_\infty$ with $p_{a_j} \to p$ and $\lim_{j \to \infty} \min_{|w|\leq \delta}\frac{\psi_{a_j} -
\psi}{a_j} (w) =
\liminf_{a \to 0+} \min_{|w|\leq \delta}\frac{\psi_a - \psi}a (w)$, we deduce that
\begin{align*}
  g'(0) + F(p, \liminf_{a \to 0+} \min_{|w|\leq \delta}\frac{\psi_a - \psi}a (w)) \leq 0.
\end{align*}
Since $\frac{\psi_a - \psi}a \to \Lambda_0[\psi](w)$ in $L^2(B_\delta(0))$, we conclude that
\begin{align*}
\essinf_{|w| \leq
\delta}
\Lambda_0[\psi](w) \geq \liminf_{a \to 0+} \min_{|w|\leq \delta}\frac{\psi_a - \psi}a (w).
\end{align*}
The monotonicity of $F$ therefore yields
\begin{align*}
  g'(0) + F(p, \essinf_{|w| \leq \delta} \Lambda_0[\psi](w)) \leq 0.
\end{align*}
Finally, we recall that we can assume that $|p| \leq \norm{\nabla \psi}_\infty$ is arbitrarily
small. Continuity of $F$ in the first variable yield the final conclusion that
\eqref{visc_subsolution} is satisfied.

The full argument is more technically involved. In particular, we need to decompose the space $\Rn$
into the direct sum of $Z_1$ and $Z_2$ as explained in Section~\ref{sec:visc-sol}, and treat them
independently. In essence, the resolvents $\frac{\psi_a - \psi}a$ do not depend
on the directions parallel to $Z_2$, \cite[Lemma~3.9]{GP16}, and therefore we can treat these
directions as we treat time in the above simplified argument.
A symmetric argument implies that $u(x,t) = \liminf_{(y,s,m) \to (x, t, \infty)} u_m$  is a
viscosity supersolution.

\medskip

To address the stability with respect to positively one-homogeneous approximations $W_m$ in (b)
above, we approximate $W_m$ by a sequence $W_m^\delta$ as in (a) and modify the previous stability
argument. In particular, we have solutions $u^\delta_m$ besides $u_m$, and $u_m^\delta \rightrightarrows
u_m$ locally uniformly as $\delta \to 0+$ due to the stability results in \cite{GGP14JMPA,GGP13AMSA}.

\subsubsection{Existence}
Existence of solutions of the equation \eqref{lse} is now a standard consequence of the stability
property. Fixing initial data $u_0$ and an approximating sequence of one-homogeneous $W_m$ so that the stability
result Theorem~\ref{th:stability-intro} holds, we get a sequence of viscosity solutions of
\eqref{lse} with the anisotropy $W_m$ and initial data $u_0$. The existence of such solutions
follows from the standard theory of viscosity solutions \cite{CGG,G06} if $F$ comes from the level
set formulation of \eqref{curvature-flow}, or from \cite{GGP13AMSA} in the general case. By the stability result
Theorem~\ref{th:stability-intro}, the half-relaxed limits $\bar u(x, t) = \limsup_{(y, s, m) \to
(x, t, \infty)} u_m(y,s)$ and $\underline u(x, t) = \liminf_{(y, s, m) \to (x, t, \infty)}
u_m(y,s)$ are respectively a subsolution and a supersolution of \eqref{lse}, without the initial
data. Clearly $\underline u \leq \bar u$. To show the other inequality, we use the comparison
theorem Theorem~\ref{th:comparison-principle-intro} after we show that $\bar u$ and $\underline u$ attain
the correct initial data $u_0$ and that they are equal to a constant outside of a compact set at
each time. This can be done via the comparison principle for the regularized problems in a rather standard way using translations of the barriers for the
solutions $u_m$ of the type
$(x, t) \mapsto \min(\max(W_m^\circ(x), c_1), c_2) + c t$ for appropriate constants $c_1, c_2, c$, where $W_m^\circ$ is the polar of $W_m$.
Since $W_m \to W$ locally uniformly, $W_m^\circ \to W^\circ$ locally uniformly as well. The existence part of the
well-posedness theorem, Theorem~\ref{th:well-posedness-intro}, is established. The uniqueness is a
direct consequence of the comparison principle, Theorem~\ref{th:comparison-principle-intro}.

\subsection*{Acknowledgments}
The work of the first author is partly supported by Japan
Society for the Promotion of Science (JSPS) through grants No. 26220702
(Kiban S) and No. 16H03948 (Kiban B).
The work of the second author is partially supported by JSPS KAKENHI Grant No. 26800068 (Wakate B).


\section*{References}
\begin{biblist}

\bib{ACM}{book}{
   author={Andreu-Vaillo, Fuensanta},
   author={Caselles, Vicent},
   author={Maz{\'o}n, Jos{\'e} M.},
   title={Parabolic quasilinear equations minimizing linear growth
   functionals},
   series={Progress in Mathematics},
   volume={223},
   publisher={Birkh\"auser Verlag, Basel},
   date={2004},
   pages={xiv+340},
   isbn={3-7643-6619-2},
   review={\MR{2033382 (2005c:35002)}},
   doi={10.1007/978-3-0348-7928-6},
}

\bib{AG89}{article}{
   author={Angenent, Sigurd},
   author={Gurtin, Morton E.},
   title={Multiphase thermomechanics with interfacial structure. II.\
   Evolution of an isothermal interface},
   journal={Arch. Rational Mech. Anal.},
   volume={108},
   date={1989},
   number={4},
   pages={323--391},
   issn={0003-9527},
   review={\MR{1013461 (91d:73004)}},
   doi={10.1007/BF01041068},
}

\bib{Anzellotti}{article}{
   author={Anzellotti, Gabriele},
   title={Pairings between measures and bounded functions and compensated
   compactness},
   journal={Ann. Mat. Pura Appl. (4)},
   volume={135},
   date={1983},
   pages={293--318 (1984)},
   issn={0003-4622},
   review={\MR{750538 (85m:46042)}},
   doi={10.1007/BF01781073},
}

\bib{Attouch}{book}{
   author={Attouch, H.},
   title={Variational convergence for functions and operators},
   series={Applicable Mathematics Series},
   publisher={Pitman (Advanced Publishing Program)},
   place={Boston, MA},
   date={1984},
   pages={xiv+423},
   isbn={0-273-08583-2},
   review={\MR{773850 (86f:49002)}},
}

\bib{B10}{article}{
   label={B1},
   author={Bellettini, G.},
   title={An introduction to anisotropic and crystalline mean curvature flow},
   journal={Hokkaido Univ. Tech. Rep. Ser. in Math.},
   volume={145},
   date={2010},
   pages={102--162},
}

\bib{BCCN06}{article}{
   author={Bellettini, Giovanni},
   author={Caselles, Vicent},
   author={Chambolle, Antonin},
   author={Novaga, Matteo},
   title={Crystalline mean curvature flow of convex sets},
   journal={Arch. Ration. Mech. Anal.},
   volume={179},
   date={2006},
   number={1},
   pages={109--152},
   issn={0003-9527},
   review={\MR{2208291 (2007a:53126)}},
   doi={10.1007/s00205-005-0387-0},
}

\bib{BCCN09}{article}{
   author={Bellettini, Giovanni},
   author={Caselles, Vicent},
   author={Chambolle, Antonin},
   author={Novaga, Matteo},
   title={The volume preserving crystalline mean curvature flow of convex
   sets in $\mathbb R^N$},
   journal={J. Math. Pures Appl. (9)},
   volume={92},
   date={2009},
   number={5},
   pages={499--527},
   issn={0021-7824},
   review={\MR{2558422 (2011b:53155)}},
   doi={10.1016/j.matpur.2009.05.016},
}

\bib{BGN00}{article}{
   author={Bellettini, G.},
   author={Goglione, R.},
   author={Novaga, M.},
   title={Approximation to driven motion by crystalline curvature in two
   dimensions},
   journal={Adv. Math. Sci. Appl.},
   volume={10},
   date={2000},
   number={1},
   pages={467--493},
   issn={1343-4373},
   review={\MR{1769163 (2001i:53109)}},
}

\bib{BN00}{article}{
   author={Bellettini, G.},
   author={Novaga, M.},
   title={Approximation and comparison for nonsmooth anisotropic motion by
   mean curvature in ${\bf R}^N$},
   journal={Math. Models Methods Appl. Sci.},
   volume={10},
   date={2000},
   number={1},
   pages={1--10},
   issn={0218-2025},
   review={\MR{1749692 (2001a:53106)}},
   doi={10.1142/S0218202500000021},
}

\bib{BNP99}{article}{
   author={Bellettini, G.},
   author={Novaga, M.},
   author={Paolini, M.},
   title={Facet-breaking for three-dimensional crystals evolving by mean
   curvature},
   journal={Interfaces Free Bound.},
   volume={1},
   date={1999},
   number={1},
   pages={39--55},
   issn={1463-9963},
   review={\MR{1865105 (2003i:53099)}},
   doi={10.4171/IFB/3},
}

\bib{BNP01a}{article}{
   author={Bellettini, G.},
   author={Novaga, M.},
   author={Paolini, M.},
   title={On a crystalline variational problem. I. First variation and
   global $L^\infty$ regularity},
   journal={Arch. Ration. Mech. Anal.},
   volume={157},
   date={2001},
   number={3},
   pages={165--191},
   issn={0003-9527},
   review={\MR{1826964 (2002c:49072a)}},
   doi={10.1007/s002050010127},
}

\bib{BNP01b}{article}{
   author={Bellettini, G.},
   author={Novaga, M.},
   author={Paolini, M.},
   title={On a crystalline variational problem. II. $BV$ regularity and
   structure of minimizers on facets},
   journal={Arch. Ration. Mech. Anal.},
   volume={157},
   date={2001},
   number={3},
   pages={193--217},
   issn={0003-9527},
   review={\MR{1826965 (2002c:49072b)}},
   doi={10.1007/s002050100126},
}

\bib{BNP01IFB}{article}{
   author={Bellettini, G.},
   author={Novaga, M.},
   author={Paolini, M.},
   title={Characterization of facet breaking for nonsmooth mean curvature
   flow in the convex case},
   journal={Interfaces Free Bound.},
   volume={3},
   date={2001},
   number={4},
   pages={415--446},
   issn={1463-9963},
   review={\MR{1869587 (2002k:53127)}},
   doi={10.4171/IFB/47},
}

\bib{BP96}{article}{
   author={Bellettini, G.},
   author={Paolini, M.},
   title={Anisotropic motion by mean curvature in the context of Finsler
   geometry},
   journal={Hokkaido Math. J.},
   volume={25},
   date={1996},
   number={3},
   pages={537--566},
   issn={0385-4035},
   review={\MR{1416006 (97i:53079)}},
   doi={10.14492/hokmj/1351516749},
}

\bib{B78}{book}{
  label={B2},
   author={Brakke, Kenneth A.},
   title={The motion of a surface by its mean curvature},
   series={Mathematical Notes},
   volume={20},
   publisher={Princeton University Press, Princeton, N.J.},
   date={1978},
   pages={i+252},
   isbn={0-691-08204-9},
   review={\MR{485012 (82c:49035)}},
}

\bib{Br71}{article}{
  label={B3},
   author={Br{\'e}zis, Ha{\"{\i}}m},
   title={Monotonicity methods in Hilbert spaces and some applications to
   nonlinear partial differential equations},
   conference={
      title={Contributions to nonlinear functional analysis},
      address={Proc. Sympos., Math. Res. Center, Univ. Wisconsin, Madison,
      Wis.},
      date={1971},
   },
   book={
      publisher={Academic Press, New York},
   },
   date={1971},
   pages={101--156},
   review={\MR{0394323}},
}

\bib{Br73}{book}{
  label={B4},
   author={Br{\'e}zis, H.},
   title={Op\'erateurs maximaux monotones et semi-groupes de contractions
   dans les espaces de Hilbert},
   language={French},
   note={North-Holland Mathematics Studies, No. 5. Notas de Matem\'atica
   (50)},
   publisher={North-Holland Publishing Co., Amsterdam-London; American
   Elsevier Publishing Co., Inc., New York},
   date={1973},
   pages={vi+183},
   review={\MR{0348562}},
}

\bib{CasellesChambolle06}{article}{
   author={Caselles, Vicent},
   author={Chambolle, Antonin},
   title={Anisotropic curvature-driven flow of convex sets},
   journal={Nonlinear Anal.},
   volume={65},
   date={2006},
   number={8},
   pages={1547--1577},
   issn={0362-546X},
   review={\MR{2248685 (2007d:35143)}},
   doi={10.1016/j.na.2005.10.029},
}

\bib{Chambolle}{article}{
   author={Chambolle, Antonin},
   title={An algorithm for mean curvature motion},
   journal={Interfaces Free Bound.},
   volume={6},
   date={2004},
   number={2},
   pages={195--218},
   issn={1463-9963},
   review={\MR{2079603}},
   doi={10.4171/IFB/97},
}

\bib{CMNP}{article}{
  author = {Chambolle, Antonin},
  author = {Morini, Massimiliano},
  author={Novaga, M.},
  author = {Ponsiglione, Marcello},
  title = {Existence and uniqueness for anisotropic and crystalline mean curvature flows},
  eprint={https://arxiv.org/abs/1702.03094},
}

\bib{CMP}{article}{
  author = {Chambolle, Antonin},
  author = {Morini, Massimiliano},
  author = {Ponsiglione, Marcello},
  title = {Existence and Uniqueness for a Crystalline Mean Curvature Flow},
  journal = {Communications on Pure and Applied Mathematics},
  issn = {1097-0312},
  eprint = {http://dx.doi.org/10.1002/cpa.21668},
  doi = {10.1002/cpa.21668},
}

\bib{CGG}{article}{
   author={Chen, Yun Gang},
   author={Giga, Yoshikazu},
   author={Goto, Shun'ichi},
   title={Uniqueness and existence of viscosity solutions of generalized
   mean curvature flow equations},
   journal={J. Differential Geom.},
   volume={33},
   date={1991},
   number={3},
   pages={749--786},
   issn={0022-040X},
   review={\MR{1100211 (93a:35093)}},
}

\bib{CIL}{article}{
   author={Crandall, Michael G.},
   author={Ishii, Hitoshi},
   author={Lions, Pierre-Louis},
   title={User's guide to viscosity solutions of second order partial
   differential equations},
   journal={Bull. Amer. Math. Soc. (N.S.)},
   volume={27},
   date={1992},
   number={1},
   pages={1--67},
   issn={0273-0979},
   review={\MR{1118699 (92j:35050)}},
   doi={10.1090/S0273-0979-1992-00266-5},
}

\bib{Evans89}{article}{
   author={Evans, Lawrence C.},
   title={The perturbed test function method for viscosity solutions of
   nonlinear PDE},
   journal={Proc. Roy. Soc. Edinburgh Sect. A},
   volume={111},
   date={1989},
   number={3-4},
   pages={359--375},
   issn={0308-2105},
   review={\MR{1007533 (91c:35017)}},
   doi={10.1017/S0308210500018631},
}

\bib{ES}{article}{
   author={Evans, L. C.},
   author={Spruck, J.},
   title={Motion of level sets by mean curvature. I},
   journal={J. Differential Geom.},
   volume={33},
   date={1991},
   number={3},
   pages={635--681},
   issn={0022-040X},
   review={\MR{1100206 (92h:35097)}},
}

\bib{G06}{book}{
   author={Giga, Yoshikazu},
   title={Surface evolution equations - a level set approach},
   series={Monographs in Mathematics},
   volume={99},
   note={(earlier version: Lipschitz Lecture Notes \textbf{44}, University of Bonn, 2002)},
   publisher={Birkh\"auser Verlag, Basel},
   date={2006},
   pages={xii+264},
   isbn={978-3-7643-2430-8},
   isbn={3-7643-2430-9},
   review={\MR{2238463 (2007j:53071)}},
}

\bib{GG98ARMA}{article}{
   author={Giga, Mi-Ho},
   author={Giga, Yoshikazu},
   title={Evolving graphs by singular weighted curvature},
   journal={Arch. Rational Mech. Anal.},
   volume={141},
   date={1998},
   number={2},
   pages={117--198},
   issn={0003-9527},
   review={\MR{1615520 (99j:35118)}},
}

\bib{GG1a}{article}{
   author={Giga, Mi-Ho},
   author={Giga, Yoshikazu},
   title={Stability for evolving graphs by nonlocal weighted curvature},
   journal={Comm. Partial Differential Equations},
   volume={24},
   date={1999},
   number={1-2},
   pages={109--184},
   issn={0360-5302},
   review={\MR{1671993}},
   doi={10.1080/03605309908821419},
}

\bib{GG01}{article}{
   author={Giga, Mi-Ho},
   author={Giga, Yoshikazu},
   title={Generalized motion by nonlocal curvature in the plane},
   journal={Arch. Ration. Mech. Anal.},
   volume={159},
   date={2001},
   number={4},
   pages={295--333},
   issn={0003-9527},
   review={\MR{1860050 (2002h:53117)}},
   doi={10.1007/s002050100154},
}

\bib{GGN}{article}{
   author={Giga, Mi-Ho},
   author={Giga, Yoshikazu},
   author={Nakayasu, Atsushi},
   title={On general existence results for one-dimensional singular
   diffusion equations with spatially inhomogeneous driving force},
   conference={
      title={Geometric partial differential equations},
   },
   book={
      series={CRM Series},
      volume={15},
      publisher={Ed. Norm., Pisa},
   },
   date={2013},
   pages={145--170},
   review={\MR{3156893}},
   doi={10.1007/978-88-7642-473-1\_8},
}

\bib{GGP13AMSA}{article}{
   author={Giga, Mi-Ho},
   author={Giga, Yoshikazu},
   author={Po{\v{z}}{\'a}r, Norbert},
   title={Anisotropic total variation flow of non-divergence type on a
   higher dimensional torus},
   journal={Adv. Math. Sci. Appl.},
   volume={23},
   date={2013},
   number={1},
   pages={235--266},
   issn={1343-4373},
   isbn={978-4-7625-0665-9},
   review={\MR{3155453}},
}
\bib{GGP14JMPA}{article}{
   author={Giga, Mi-Ho},
   author={Giga, Yoshikazu},
   author={Po{\v{z}}{\'a}r, Norbert},
   title={Periodic total variation flow of non-divergence type in
   $\mathbb{R}^n$},
   language={English, with English and French summaries},
   journal={J. Math. Pures Appl. (9)},
   volume={102},
   date={2014},
   number={1},
   pages={203--233},
   issn={0021-7824},
   review={\MR{3212254}},
   doi={10.1016/j.matpur.2013.11.007},
}

\bib{GP16}{article}{
   author={Giga, Yoshikazu},
   author={Po{\v{z}}{\'a}r, Norbert},
   title={A level set crystalline mean curvature flow of surfaces},
   journal={Adv. Differential Equations},
   volume={21},
   date={2016},
   number={7-8},
   pages={631--698},
   issn={1079-9389},
   review={\MR{3493931}},
}

\bib{Il93}{article}{
   author={Ilmanen, Tom},
   title={Convergence of the Allen-Cahn equation to Brakke's motion by mean
   curvature},
   journal={J. Differential Geom.},
   volume={38},
   date={1993},
   number={2},
   pages={417--461},
   issn={0022-040X},
   review={\MR{1237490 (94h:58051)}},
}

\bib{Ishii14}{article}{
   author={Ishii, Katsuyuki},
   title={An approximation scheme for the anisotropic and nonlocal mean
   curvature flow},
   journal={NoDEA Nonlinear Differential Equations Appl.},
   volume={21},
   date={2014},
   number={2},
   pages={219--252},
   issn={1021-9722},
   review={\MR{3180882}},
   doi={10.1007/s00030-013-0244-z},
}

\bib{Komura67}{article}{
   author={K{\=o}mura, Yukio},
   title={Nonlinear semi-groups in Hilbert space},
   journal={J. Math. Soc. Japan},
   volume={19},
   date={1967},
   pages={493--507},
   issn={0025-5645},
   review={\MR{0216342 (35 \#7176)}},
}

\bib{LMM}{article}{
  author={Lasica, Michal},
  author={Moll, Salvador},
  author={Mucha, Piotr B.},
  title={Total variation denoising in $l^1$ anisotropy},
  journal={arXiv:1611.03261},
  status={preprint},
  eprint={https://arxiv.org/abs/1611.03261},
}


\bib{Moll}{article}{
   author={Moll, J. S.},
   title={The anisotropic total variation flow},
   journal={Math. Ann.},
   volume={332},
   date={2005},
   number={1},
   pages={177--218},
   issn={0025-5831},
   review={\MR{2139257 (2006d:35113)}},
   doi={10.1007/s00208-004-0624-0},
}

\bib{OS}{article}{
   author={Osher, Stanley},
   author={Sethian, James A.},
   title={Fronts propagating with curvature-dependent speed: algorithms
   based on Hamilton-Jacobi formulations},
   journal={J. Comput. Phys.},
   volume={79},
   date={1988},
   number={no.~1},
   pages={12--49},
   issn={0021-9991},
   review={\MR{965860}},
}

\bib{Rockafellar}{book}{
   author={Rockafellar, R. Tyrrell},
   title={Convex analysis},
   series={Princeton Mathematical Series, No. 28},
   publisher={Princeton University Press, Princeton, N.J.},
   date={1970},
   pages={xviii+451},
   review={\MR{0274683 (43 \#445)}},
}

\bib{S}{article}{
   author={Soner, Halil Mete},
   title={Motion of a set by the curvature of its boundary},
   journal={J. Differential Equations},
   volume={101},
   date={1993},
   number={2},
   pages={313--372},
   issn={0022-0396},
   review={\MR{1204331}},
   doi={10.1006/jdeq.1993.1015},
}

\bib{TT}{article}{
   author={Takasao, K.},
   author={Tonegawa, Y.},
   title={Existence and regularity of mean curvature flow with transport term in higher dimension},
   status={to appear in Math. Annalen},
   eprint={http://arxiv.org/abs/1307.6629},
}

\bib{T91}{article}{
   author={Taylor, Jean E.},
   title={Constructions and conjectures in crystalline nondifferential
   geometry},
   conference={
      title={Differential geometry},
   },
   book={
      title={Proceedings of the Conference on Differential Geometry, Rio de Janeiro},
      editor={Lawson, B.},
      editor={Tanenblat, K.},
      series={Pitman Monogr. Surveys Pure Appl. Math.},
      volume={52},
      publisher={Longman Sci. Tech., Harlow},
   },
   date={1991},
   pages={321--336},
   review={\MR{1173051 (93e:49004)}},
   doi={10.1111/j.1439-0388.1991.tb00191.x},
}

\end{biblist}
\end{document}